\documentclass[11pt]{amsart}
\usepackage{amsmath,amssymb,latexsym,esint,cite}
\usepackage{verbatim}
\usepackage[left=2.9cm,right=2.9cm,top=3.2cm,bottom=3.2cm]{geometry}

\usepackage{color,enumitem,graphicx}
\usepackage[colorlinks=true,urlcolor=blue, citecolor=red,linkcolor=blue,linktocpage,pdfpagelabels, bookmarksnumbered,bookmarksopen]{hyperref}

\usepackage[hyperpageref]{backref}
\usepackage[english]{babel}

\newcommand{\N}{\mathbb N}

\newcommand{\pnorm}[2][]{\if #1'' \left|#2\right|_p \else \left|#2\right|_{#1} \fi}
\newcommand{\R}{\mathbb R}

\newcommand{\I}{\mathbb I}
\newcommand{\J}{\mathbb J}

\newcommand{\RN}{{\mathbb{R}^N}}
\newcommand{\CA}{{\mathcal A}}
\newcommand{\CB}{{\mathcal B}}
\newcommand{\CC}{{\mathcal C}}
\newcommand{\HN}{H^1}
\newcommand{\LN}{L^2}
\newcommand{\LI}{L^{\infty}(\RN)}

\newtheorem{lemma}{Lemma}[section]
\newtheorem{proposition}[lemma]{Proposition}
\newtheorem{theorem}[lemma]{Theorem}
\newtheorem{remark}[lemma]{Remark}
\newtheorem{corollary}[lemma]{Corollary}

\newcommand{\ef}{\eqref}

\numberwithin{equation}{section}

\title[Uniqueness of limit flow for quasi-linear parabolic equations]{Uniqueness 
	of limit flow for a class \\ of quasi-linear parabolic equations}

\author[M.\ Squassina]{Marco Squassina}
\author[T.\ Watanabe]{Tatsuya Watanabe}

\address[M.\ Squassina]{Dipartimento di Informatica
	\newline\indent
	Universit\`a degli Studi di Verona
	\newline\indent
	Verona, Italy}
\email{marco.squassina@univr.it}

\address[T. Watanabe]{Department of Mathematics
	\newline\indent	
	Faculty of Science, Kyoto Sangyo University
		\newline\indent
	 Motoyama, Kamigamo, Kita-ku, Kyoto-City, Japan}
\email{tatsuw@cc.kyoto-su.ac.jp}


\thanks{This paper was carried out while the second author was staying at
	University Bordeaux I. The author is very grateful to all the staff of 
	University Bordeaux I for their kind hospitality.
	The first author is supported by Gruppo Nazionale per l'Analisi Matematica,
	la Probabilit\`a e le loro Applicazioni (INdAM).
	The second author is supported by 
	JSPS Grant-in-Aid for Scientific Research (C) (No. 15K04970).}

\begin{document}

\begin{abstract}
We investigate the issue of uniqueness of the limit flow for a relevant class of quasi-linear parabolic equations defined on the whole space. 
More precisely, we shall investigate conditions which guarantee that the global solutions decay at infinity uniformly in time and their entire trajectory approaches a single steady state as time goes to infinity. 
Finally, we obtain a characterization of solutions which blow-up, vanish or converge to a stationary state for initial data of the form $\lambda \varphi_0$ while $\lambda>0$ crosses a bifurcation value $\lambda_0$.
\end{abstract}

\maketitle

\begin{center}
	\begin{minipage}{8cm}
		\small
		\tableofcontents
	\end{minipage}
\end{center}

\section{Introduction and main results}

\subsection{Overview}
In the last decades, a considerable attention has been devoted to the study of solutions 
to the quasi-linear Schr\"odinger equation
\begin{equation}
\label{eq:1.0}
{\rm i} u_{t}+\Delta u+u \Delta u^2 -u +|u|^{p-1}u=0 \quad \text{in $\R^N\times (0,\infty)$}.
\end{equation}
In fact, this equation arises in superfluid film equation in plasma physics, see \cite{Briz1,BEPZ},
and it is also a more accurate model in a many physical phenomena compared with
the classical semi-linear Schr\"odinger equation ${\rm i} u_{t}+\Delta u-u+|u|^{p-1}u=0$. 
In particular, local well-posedness, regularity, existence and properties of ground states 
as well as stability of standing wave solutions were investigated, 
see e.g.\ \cite{CJS} and the references therein. 
The problem raised the attention also in the framework of non-smooth critical point theory, 
since the functional associated with the standing wave solutions of \eqref{eq:1.0}
$$
u\mapsto   \frac{1}{2} \int_{\RN} (1+2u^2) |\nabla u|^2dx +\frac{1}{2} \int_{\RN} u^2dx
-\frac{1}{p+1} \int_{\RN} |u|^{p+1} \,dx,
$$
is merely lower semi-continuous on the Sobolev space $H^1(\R^N)$ 
and it turns out that it is differentiable only along bounded directions. 
Hence on $H^1(\R^N)$, the existence of critical points required the development of new tools 
and ideas, see e.g.\ \cite{pelsqu,solfsqu} and \cite{liuwang}.

In this paper, motivated by the results obtained in \cite{delpino-cort,CGH} 
for a class of semi-linear parabolic equations, 
we aim to investigate the asymptotic behavior for the quasi-linear parabolic problem
\begin{align}
\label{eq:1.1}
 u_{t}-\Delta u-u \Delta u^2 +u =|u|^{p-1}u \quad & \text{in $\R^N\times (0,\infty)$}, \\
\label{eq:1.2}
 u(x,0)=u_0(x) \quad & \text{in $\R^N$},
\end{align}
whose corresponding stationary problem is
\begin{align}
\label{eq:1.3}
-\Delta u -u \Delta u^2+u &=|u|^{p-1}u \quad  \text{in $\R^N$}, \\
u(x) \to 0 \quad & \text{as $|x| \to \infty.$}   \notag
\end{align}
More precisely, we deal with the problem of {\em uniqueness} of the limit 
of bounded trajectories of \eqref{eq:1.1}-\eqref{eq:1.2}. 
Since the problem is invariant under translations, 
even knowing that \eqref{eq:1.3} admits a unique solution up to translations 
in general does not prevent from having different positively diverging sequences 
$\{t_n\}_{n\in\N}$, $\{\tau _n\}_{n\in\N},$ such that 
$\{u(\cdot,t_n)\}_{n\in\N}$ and $\{u(\cdot,\tau_n)\}_{n\in\N}$ 
converge to different solutions to \eqref{eq:1.3}. 
As proved by L.~Simon in a celebrated paper \cite{simon} (see also \cite{jen}), 
in the case of variational parabolic problems such as $u_t+{\mathcal E}'(u,\nabla u)=0$ 
where the associated Lagrangian ${\mathcal E}(s,\xi)$ depends {\em analytically} 
on its variables $(s,\xi)$, 
then it is always the case that the full flow $u(t)$ converges to a stationary solution 
of ${\mathcal E}'(u,\nabla u)=0 $ and oscillatory behavior is thus ruled out. 
The argument is essentially based upon Lojasiewicz inequality \cite{loj} 
and a series of additional estimates. 
On the other hand for \eqref{eq:1.1}, the assumptions of \cite{simon} are not fulfilled 
due to the presence of the {\em non-analytical} nonlinearity $u\to |u|^{p-1}u$, 
unless $p$ is an odd integer. 
In general, without the analyticity assumption, 
the $\omega$-limit set corresponding to a suitable sub-manifold of initial data 
is a {\em continuum} of $H^1$ which is homeomorphic to the sphere, see \cite{polacik2,polacik1}. 
However, equation \eqref{eq:1.3} has been object of various investigations 
for what concerns uniqueness and {\em non-degeneracy} of solutions. 
By working on the linearized operator ${\mathcal L}$ around a stationary solution $w$, namely
\begin{equation} \label{linearized}
{\mathcal L} \phi= -(1+2w^2)\Delta \phi -4w \nabla w \cdot \nabla \phi 
-(4w\Delta w+2| \nabla w|^2) \phi +\phi -p|w|^{p-1} \phi,
\end{equation}
and by exploiting the non-degeneracy \cite{ASW1,S} of the positive radial solutions 
to \eqref{eq:1.3}, i.e. 
$$
{\rm Ker} ({\mathcal L})= {\rm span} \Big\{\frac{\partial w}{\partial x_1},\ldots,\frac{\partial w}{\partial x_N}\Big\},
$$
inspired by the ideas of \cite{delpino-cort} where the semi-linear case is considered, 
we will be able to prove that, in fact, 
the flow of \eqref{eq:1.1}-\eqref{eq:1.2} enjoys uniqueness. 
As to similar results for semi-linear parabolic problems, 
see \cite{BJP, CGH, FP} and references therein.
Throughout the paper we shall assume that
$$
3\le \, p<\frac{3N+2}{N-2} \quad\text{if $N \ge 3$},\qquad 
3 \le p<\infty \quad\text{if $N=1,2$}.
$$
We will deal with classical solutions $u\in C([0,T_0),C^2(\R^N))\cap C^1((0,T_0),C(\R^N))$
to \ef{eq:1.1}-\ef{eq:1.2}, whose local existence and additional properties will be established in Section 2. 
The uniqueness of positive solutions of \ef{eq:1.3} has been investigated in \cite{AW,GS}, 
while the non-degeneracy of the unique positive solution has been also obtained 
in \cite{ASW1, ASW2, S}. 
We also note that the unique positive solution $w$ of \ef{eq:1.3} is radially symmetric
with respect to a point $x_0\in \RN$ and decays exponentially at infinity.
For a good source of references for the issue of long term behavior of semi-linear parabolic equations, we refer the reader to \cite{gazzolaweth}.

\vskip4pt
\subsection{Main results}
The followings are the main results of the paper.

\begin{theorem}[Decaying solutions]
	\label{thm:1.2}
Let $N \ge 2$ and let $u_0 \in C_0^{\infty}(\R^N)$ be non-negative and radially non-increasing. 
Let $u$ be the corresponding solution to \ef{eq:1.1}-\ef{eq:1.2} and 
assume that it is globally defined. 
Then $u$ is positive, bounded, radially decreasing and 
	\begin{equation}\label{eq:1.4}
	\lim_{|x| \to \infty} \sup_{ t>0} u(x,t)=0.
	\end{equation}
\end{theorem}

\begin{theorem}[Uniqueness of limit]
	\label{thm:1.1}
Let $N \ge 1$ and let $u$ be a non-negative, bounded, globally defined 
solution to \ef{eq:1.1}-\ef{eq:1.2} which satisfies \ef{eq:1.4}. 
Then either $u(x,t) \to 0$ uniformly in $\RN$ as $t \to \infty$ or 
there is a positive solution $w$ of \ef{eq:1.3} such that $u(x,t) \to w(x)$ uniformly in $\RN.$
In addition, 
\begin{equation}
\label{eq:1.5}
\lim_{t \to \infty} 
\int_0^K \| u(\cdot,t+s)-w(\cdot) \|_{H^1(\RN)}^2 \,ds =0,
\end{equation}
for every $K>0$.
\end{theorem}

\begin{theorem}[Bifurcation]
	\label{thm:1.3}
Let $N \ge 2$ and let $\varphi_0 \in C_0^{\infty}(\R^N)$ be 
non-negative, radially non-increasing and not identically equal to zero.
If $p=3$, assume furthermore that
$$
\int_{\RN} \Big(\varphi_0^2 |\nabla \varphi_0|^2
-\frac{1}{4} |\varphi_0|^{4} \Big) \,dx <0.
$$ 
Then there exists $\lambda_0>0$ such that the solution $u$ 
to \ef{eq:1.1}-\ef{eq:1.2} with $u_0=\lambda \varphi_0$ satisfies
\begin{itemize}
\item[\rm(i)] If $\lambda<\lambda_0$, 
then $u(x,t)$ goes to zero as $t \to \infty$ uniformly in $\RN$.
\item[\rm(ii)] If $\lambda=\lambda_0$, 
then $u(x,t)$ converges to a positive solution $w$ of \ef{eq:1.3} uniformly in $\RN$.
\item[\rm(iii)] If $\lambda>\lambda_0$, then $u(x,t)$ blows up in finite time.
\end{itemize}
\end{theorem}

\medskip
\noindent
{\bf Remark.} Here we collect some remarks on the main results.

\noindent
(i) In Theorem \ref{thm:1.1}, we don't need any symmetric assumptions on 
the solution. However by the result in \cite{P}, we can show that our global solution is
{\em asymptotically symmetric}, that is, it has a common center of symmetry for the 
elements of the $\omega$-limit.
See \cite{BJP, MSS} for related results.

\noindent
(ii) To prove the uniform decay condition \ef{eq:1.4} in Theorem \ref{thm:1.2}, 
we have to assume that $u_0$ is radially non-increasing. 
This assumption is used to obtain a universal bound near infinity, see Remark \ref{rem:6.3}.
We believe that this is technical, but we don't know how to remove it at present.

\noindent
(iii) By a recent result in \cite{ASW2}, the non-degeneracy of the positive radial solution to 
\ef{eq:1.3} holds even if $1<p<3$. On the other hand, 
the condition $p \ge 3$ appears in various situations, 
especially in the proof of Theorem \ref{thm:1.2}. 
Although the nonlinear term $|u|^{p-1}u$ is superlinear even when $1<p<3$, 
problem \ef{eq:1.1} has a {\em sublinear structure} due to the term $u \Delta u^2$, 
causing our arguments to completely fail.

\noindent
(iv) In the proof of Theorem \ref{thm:1.2}, we also require that $N \ge 2$.
This is to construct a suitable supersolution, see Remark \ref{rem:6.5}.

\medskip
\noindent
As we will see in Section 2, our problem is uniformly parabolic, yielding that basic tools 
(energy estimate, Schauder estimate, Comparison Principle, etc) are available. 
Especially some proofs work in the spirit of those of \cite{delpino-cort} 
for semi-linear problems. 
However quite often the semi-linear techniques fail to work, 
especially in the construction of suitable sub-solutions (see e.g.\ Lemma~\ref{lem:6.2}). 
To compare the dynamical behavior of solutions for our quasi-linear parabolic problem 
with that for the corresponding semi-linear one, for $\kappa>0$, we consider the problem
\begin{equation}
\label{eq:7.3}
\left\{
\begin{array}{rl}
u_t-\Delta u-\kappa u\Delta u^2 +u=|u|^{p-1}u &\ \hbox{in} \ \RN \times (0,\infty),\\
u(x,0)=\lambda \varphi_0(x) &\ \hbox{in} \ \RN
\end{array}
\right.
\end{equation}
and the corresponding semi-linear parabolic problem:
\begin{equation}\label{eq:7.4}
\left\{
\begin{array}{rl}
u_t-\Delta u+u=|u|^{p-1}u &\ \hbox{in} \ \RN \times (0,\infty),\\
u(x,0)= \lambda \varphi_0(x) &\ \hbox{in} \ \RN.
\end{array}
\right.
\end{equation}
The stationary problem associated with \ef{eq:7.4} is given by
\begin{equation}\label{eq:7.5}
-\Delta w+w=|w|^{p-1}w \ \hbox{in} \ \RN, \quad w(x) \to 0 \ \hbox{as} \ |x| \to \infty.
\end{equation}
It is well-known that problem \ef{eq:7.5} has a unique positive solution 
for $1<p<(N+2)/(N-2)$ if $N \ge 3$ and $1<p<\infty$ if $N=1,2$. 
Now let $\lambda_0(\kappa)>0$ be a constant obtained by applying 
Theorem \ref{thm:1.3} to \ef{eq:7.3}. 
When $3 \le p< (N+2)/(N-2)$, both $\lambda_0(\kappa)$ and $\lambda_0(0)$ are defined and
\begin{equation}\label{eq:7.6}
\lambda_0(0) < \lambda_0(\kappa) \quad \hbox{for all} \ \kappa>0.
\end{equation}
In fact we claim that $\lambda_0(\kappa_0) < \lambda_0(\kappa_1)$ for all $\kappa_0<\kappa_1$. 
Defining $I_{\kappa}$ by
$$
I_{\kappa}(u)=\frac{1}{2} \int_{\RN} \big( (1+2\kappa u^2) |\nabla u|^2 +u^2 \big) \,dx
-\frac{1}{p+1} \int_{\RN} |u|^{p+1} \,dx,
$$
it follows that if $I_{\kappa_0}(u) \ge 0$ for all $u \in C_0^{\infty}(\RN)$, then
$I_{\kappa_1}(u) \ge 0$ for all $u \in C_0^{\infty}(\RN)$. 
Thus by Lemmas \ref{lem:6.6}-\ref{lem:6.7} and by the definition of $\lambda_0(\kappa)$, 
the claim follows. 
Inequality \ef{eq:7.6} shows that there exist initial values $u_0$ such that
the corresponding solution to \eqref{eq:7.3} is globally defined, 
but that of the semi-linear problem \eqref{eq:7.4} blows up in finite time.
In other words, the quasi-linear term $u\Delta u^2$ prevents blow-up of solutions. 
This kind of {\em stabilizing effects} has 
been observed for the quasi-linear Schr\"odinger equation \ef{eq:1.0}, see e.g.\ \cite{BEPZ,CJS}.

\medskip
\noindent
{\bf Plan of the paper.} 
In Section 2, we state several preparatory results.
In Subsection 2.1, we establish the local existence of classical solutions of \ef{eq:1.1} 
and give qualitative properties of classical solutions. 
Subsection 2.2 concerns with stability estimates for global solutions.
We prove uniform estimates of global solutions in Subsection 2.3.
We state technical results about uniqueness of limit in Subsection 2.4.
In Section 3, we will prove the main results of the paper. 

\medskip
\noindent
{\bf Notations.} {\small 
For any $p\in [1,\infty)$ and a domain $U \subset \RN$, 
the space $L^p(U)$ is endowed with the norm 
$$
\|u\|_{L^p(U)}=\Big(\int_{U} |u|^p \,dx \Big)^{1/p}\!\!\!.
$$
$(\cdot, \cdot)_{L^2(U)}$ denotes the standard inner product in $L^2(U)$.
The Sobolev space $H^1(U)$ is endowed with the standard norm 
$$
\|u\|_{H^1(U)}=\Big(\int_{U} \big( |\nabla u|^2 +|u|^2 \big) \,dx \Big)^{1/2}\!\!.
$$ 
The higher order spaces $H^m(U)$ are endowed with the standard norm. 
The space $C^k \big( (0,T),H^m(U) \big)$ denotes the functions with $k$ time derivatives 
which belong to $H^m(U)$. When $U=\RN$, we may write $\| \cdot \|_{H^m(\RN)}=\| \cdot \|_{H^m}$. 
The symbols ${\partial u}/{\partial x_i}$, $\partial^2 u/{\partial x_i \partial x_j}$ 
and $u_t$ denote, respectively, the first and second order space derivatives and 
the time derivative of $u$. 
For non-negative integer $m$, $D^{m}u$ denotes the set of all partial derivatives of order $m$. 
$C^\infty_0(\R^N)$ denotes the space of compactly supported smooth functions. 
The notation ${\rm span}\{w_1,\ldots, w_k\}$ denotes the vector space generated by the vectors $\{w_1,\ldots, w_k\}$.
We denote by $\Omega(u)$ the $\omega$-limit set of $u$, namely the set
$$
\Omega(u):=\big\{ w\in H^1(\R^N) : \, u(\cdot,t_n) \to w \ \hbox{uniformly in $\RN$ as $n\to\infty$, for some} \ t_n \to \infty \big\}.
$$
The symbol $B(x_0,R)$ denotes a ball in $\R^N$ of center $x_0$ and with radius $R$. 
The complement of a measurable set $E\subset\R^N$ will be denoted by $E^c$.
}

\section{Preparatory results}

\subsection{Local existence and basic properties}
In this subsection, we prove the local existence of classical solutions of
\ef{eq:1.1}-\ef{eq:1.2} and provide also some qualitative properties.
First we observe that \ef{eq:1.1} can be written as $L(u)=0$ where
\begin{equation*}
L(u) = (1+2u^2)\Delta u+2u |\nabla u|^2 -u+|u|^{p-1}u - u_t
=: F\left( u, \frac{\partial u}{\partial x_i}, \frac{\partial^2 u}{\partial x_i \partial x_j} \right)-u_t, \nonumber 
\end{equation*}
where we have set
$$
F(u, p_i, r_{ij})=
\sum_{i,j=1}^N (1+2u^2) \delta_{ij} r_{ij}
+2u \sum_{i=1}^N p_i^2 -u +|u|^{p-1}u.$$
Then one has $\frac{\partial F}{\partial r_{ij}}=(1+2u^2)\delta_{ij}$ and hence
$$
\sum_{i,j=1}^N \frac{\partial F}{\partial r_{ij}}
\left( u, \frac{\partial u}{\partial x_i}, \frac{\partial^2 u}{\partial x_i \partial x_j} \right) 
\xi_i \xi_j = (1+2u^2)| \xi|^2 \ge |\xi|^2,
$$
for all $\xi \in \RN \setminus \{0\}$ and $u\in \R$.
This implies that $F$ is uniformly elliptic and the nonlinear operator $L$ is
(strongly) parabolic with respect to any $u$. 
We also note that $L$ can be written by the divergence form
$$
L(u)={\rm div} A(u,\nabla u) +B(u,\nabla u) -u_t,$$
\begin{equation} \label{divergence}
A(u, {\bf p})=(1+2u^2) {\bf p}, \quad
B(u, {\bf p})=-(1+2 | {\bf p}|^2)u+|u|^{p-1}u.
\end{equation}
Then we have the following result on the local existence of classical solutions 
whose proof is based on a {\it modified Galerkin method} as in \cite{T}.

\begin{lemma}[Local existence]
	\label{lem:0.1}
Let $u_0 \in C_0^{\infty}(\RN)$. 
Then there exist $T_0=T_0(u_0) \in (0,\infty]$ and a unique classical solution 
$u(x,t)$ of \ef{eq:1.1}-\ef{eq:1.2} satisfying 
\begin{align}\label{eq:0.1}
\sup_{t\in (0,T_0)} \| D^ku(\cdot ,t) \|_{L^{\infty}(\RN)} < \infty 
\ \hbox{for} \ |k| \le 2, \\
\label{eq:0.2}
u(x,t) \to 0 \ \hbox{as} \ |x| \to \infty \ \hbox{for each} \ t\in (0,T_0).
\end{align}
\end{lemma}

\begin{proof}
Since the operator $L$ is strongly parabolic, for any $u_0 \in C_0^{\infty}(\RN)$, 
there exist a (small) positive number $T_0=T_0(u_0)$ and a unique solution 
$u(x,t)$ of \ef{eq:1.1}-\ef{eq:1.2} satisfying
$$
u\in C \left( [0,T_0), H^m(\RN) \right)
\cap C^1 \left( (0,T_0),H^{m-2}(\RN) \right)
\ \hbox{for any} \ m\in \N \ \hbox{with} \ m>\frac{N}{2}+2
$$
by using a suitable approximation and applying the energy estimate,
see \cite[Proposition 7.5]{T}.
Then by the Sobolev embedding $H^m(\RN) \hookrightarrow C^2(\RN)$
and $H^{m-2}(\RN) \hookrightarrow C(\RN)$ for $m>\frac{N}{2}+2$, 
$u$ is a classical solution of \ef{eq:1.1}-\ef{eq:1.2}. 
Moreover by the Sobolev and Morrey inequalities, \ef{eq:0.1} and \ef{eq:0.2} also hold. 
\end{proof}

\noindent
From \ef{divergence}, we can also obtain the local existence of classical solutions
by applying the Schauder estimate, see \cite[Theorem 8.1, p.495]{LSU}.
We note that $T_0$ is not the maximal existence lifespan, 
but the local solution $u(x,t)$ can be extended beyond $T_0$ 
as long as $\sup \| u(\cdot,t)\|_{C^2(\RN)}$ is bounded. 
Next we prepare the following Comparison Principle for later use. 
For this statement, we refer the reader to \cite[Section 7, Theorem 12, p.187]{PW}.

\begin{lemma}[Comparison principle]
	\label{lem:0.2}
Let $U$ be a bounded domain in $\RN$ and $T>0$.
Suppose that $u$ is a solution of $L(u)=f(x,t)$ in $U \times (0,T]$
satisfying the initial boundary conditions:
\begin{align*}
u(x,t) = g_1(x,t) \quad &\hbox{on} \ \partial U \times (0,T),\\
u(x,0) = g_2(x) \quad &\hbox{in} \ U.
\end{align*}
Assume that $z(x,t)$ and $Z(x,t)$ satisfy the inequalities:
\begin{align*}
L(Z) \le f(x,t) \le L(z) &\ \hbox{in} \ U \times (0,T],\\
z(x,t) \le g_1(x,t) \le Z(x,t) &\ \hbox{on} \ \partial U \times (0,T),\\
z(x,0) \le g_2(x) \le Z(x,0) &\ \hbox{in} \ U.
\end{align*}
If $L$ is parabolic with respect to the functions
$\theta u+(1-\theta)z$ and $\theta u+(1-\theta)Z$ for any $\theta \in [0,1]$, 
then it follows that
$$
z(x,t) \le u(x,t) \le Z(x,t) \ \hbox{in} \ U \times (0,T].$$
\end{lemma}

\noindent
We recall that $z$ and $Z$ are called a {\it subsolution} 
and a {\it supersolution} of $L(u)=f$ respectively.
By applying Lemma \ref{lem:0.2}, we provide some qualitative properties
for solutions of \ef{eq:1.1}-\ef{eq:1.2}.

\begin{lemma}[Radially decreasing flows]
	\label{lem:0.3}
Suppose that $u_0 \in C_0^{\infty}(\RN)$ is non-negative and not identically zero. 
Then the corresponding solution $u$ is positive for
all $x\in \RN$ and $t\in (0,T_0)$.
Moreover if $u_0(x)=u_0(|x|)$ and $u_0'(r) \le 0$ for all $r \ge 0$, 
then $u(x,t)$ is also radial and 
$u_r(r,t) < 0$ for all $r \ge 0$ and $t\in (0,T_0)$.
\end{lemma}

\begin{proof}
First since $z \equiv 0$ is a subsolution of $L(u)=0$, 
it follows by Lemma \ref{lem:0.2} that $u \ge 0$.
Moreover from \ef{divergence}, we can see that the structural assumptions
for quasi-linear parabolic equations in \cite{Tr} are fulfilled.\
Then we can use the time-dependent Harnack inequality for $L(u)=0$,
see \cite[Theorem 1.1]{Tr}. Thus we have $u>0$.\
Next we suppose that $u_0$ is radial. 
Then by the local uniqueness and the rotation invariance of problem \ef{eq:1.1}, 
it follows that $u$ is radially symmetric. 
Let us assume that $u_0'(r) \le 0$ for all $r \ge 0$. 
We show that $u_r \le 0$. To this end, we follow an idea in \cite[Section 52.5]{QS}. 
Now we differentiate \ef{eq:1.1} with respect to $r$ 
and write $u'=u_r$ for simplicity. Then by a direct calculation, one has
\begin{align*}
u_t'&-(1+2u^2)u'''- \big( 8uu'+\frac{N-1}{r}(1+2u^2) \Big)u'' \\
&- \Big( \frac{4(N-1)}{r}uu'-\frac{N-1}{r^2}(1+2u^2)
+2(u')^2+pu^{p-1}-1 \Big) u'=0.
\end{align*}
We put $\phi(r,t)=u_r(r,t)e^{-Kt}$ for $K>0$. 
Then $\phi$ satisfies the following parabolic problem:
\begin{align}\label{eq:0.3}
\tilde{L}(\phi)&:= (1+2u^2)\phi''+a\phi'+b\phi-\phi_t=0, \\
a(r,t)&=8uu'+\frac{N-1}{r}(1+2u^2), \nonumber \\
b(r,t)&= \frac{4(N-1)}{r}uu'-\frac{N-1}{r^2}(1+2u^2)
+2(u')^2+pu^{p-1}-1-K. \nonumber
\end{align}
Moreover choosing sufficiently large $K>0$, we may assume that
$b(r,t) \le 0$ in $(0,\infty) \times (0,T_0)$.
Hereafter we write $Q=(0,\infty) \times (0,T_0)$ for simplicity.
Next we suppose that 
$$\displaystyle \sup_{Q} \phi(r,t)>0.$$
Then we can take 
$$M:= \frac{1}{2} \displaystyle \sup_{Q} \phi(r,t)>0$$
and put $\Phi(r,t)=\phi(r,t)-M$. 
We observe that $\Phi(0,t)=u_r(0,t)e^{-Kt}-M=-M$
for every $t\in (0,T_0)$. 
Moreover $\Phi(r,t) \to -M$ as $r \to \infty$ for each
$t\in (0,T_0)$. 
In fact since $\| \nabla u(\cdot,t) \|_{H^{m-1}(\RN)}<\infty$ for any $m>\frac{N}{2}+2$, 
it follows by the Morrey embedding theorem that
$|\nabla u(x,t)| \to 0$ as $|x| \to \infty$ and hence
$$
\lim_{r\to\infty} \Phi(r,t)=\lim_{r\to\infty} u_r(r,t)e^{-Kt}-M= -M.
$$
Now since
$$\displaystyle \sup_{t\in (0,T_0)} \Phi(0,t)
= \lim_{r \to \infty} \sup_{t\in (0,T_0)} \Phi(r,t)=-M,$$
there exist $(r_0,r_1) \subset (0,\infty)$ such that
$\Phi(r,t) \le 0$ for $r\in (0,r_0) \cup (r_1,\infty)$ and $t\in (0,T_0)$,
\begin{equation}\label{eq:0.4}
\Phi(r_0,t)=\Phi(r_1,t)=0 \ \hbox{for} \ t\in (0,T_0).
\end{equation}
Moreover by the definition of $\Phi$ and from $u_0'(r) \le 0$, 
it follows that
\begin{equation}\label{eq:0.5}
\Phi(r,0)=\phi(r,0)-M=u_r(r,0)-M=u_0'(r)-M<0
\ \hbox{for} \ r\in (r_0,r_1).
\end{equation}
Finally from \ef{eq:0.3}, $\Phi=\phi-M$ and $b \le 0$, we also have
\begin{equation}\label{eq:0.6}
\tilde{L}(\Phi)=\tilde{L}(\phi)-bM=-bM \ge 0 \ \hbox{in} \ 
(r_0,r_1) \times (0,T_0).
\end{equation}
Since the operator $\tilde{L}$ is parabolic, we can apply the Comparison Principle. 
Thus from \ef{eq:0.4}-\ef{eq:0.6}, it follows that $\Phi$ is a subsolution of $\tilde{L}(u)=0$ 
and hence $\Phi \le 0$ in $(0,\infty) \times (0,T_0)$.
On the other hand by the definition of $M$, one has 
$$
\displaystyle \sup_{Q} \Phi(r,t)= \sup_{Q} \phi(r,t)-M=M>0.
$$ 
This is a contradiction. Thus $\sup_Q \phi \le 0$ and hence
$u_r(r,t)=e^{Kt}\phi(r,t) \le 0$ for all $r \ge 0$ and $t\in (0,T_0)$.
This completes the radial non-increase of $u$ as required.
Finally the radial decrease of $u$ follows by the Hopf lemma,
see \cite[Theorem 6, P. 174]{PW}.
\end{proof}

\subsection{Energy stabilization}
In this subsection, we prove several stability estimates. 
Let $u(x,t)$ be a non-negative bounded, globally defined solution
of \ef{eq:1.1}-\ef{eq:1.2} satisfying \ef{eq:1.4} and denote by $\Omega(u)$ the
$\omega$-limit set of $u$.
We also suppose that $\| \nabla u(\cdot,t) \|_{L^{\infty}(\RN)}$ is uniformly bounded.
We define the functional
$$
I(u):= \frac{1}{2} \int_{\RN} \big( (1+2u^2)|\nabla u|^2 +u^2 \big) \,dx
-\frac{1}{p+1} \int_{\RN} |u|^{p+1} \,dx.
$$
Notice that $I$ is well defined on the set of functions $u\in H^1(\R^N)$
such that $u^2\in H^1(\R^N)$ from $3 \le p<\frac{3N+2}{N-2}$, 
via the Sobolev embedding (cf.\ \cite{CJS}). Moreover we have the following.

\begin{lemma}[Energy identity]
	There holds
	\label{energy-est}	
	\begin{equation*}
	\frac{d}{dt} I \big( u(\cdot ,t) \big) 
	=- \int_{\RN} u_t(x,t)^2 \,dx.
	\end{equation*}
\end{lemma}
\begin{proof}
It is possible to prove that $I$ is differentiable along smooth bounded directions. 
By the proof of Lemma \ref{lem:0.1}, we know that $u \in C^1 \big( (0, T_0), H^{m-2}(\RN) \big)$ 
for any $m>N/2+2$. 
Since $H^{m-2}(\RN) \hookrightarrow L^{\infty}(\RN)$ for $m> N/2+2$, 
it follows that $u \in C^1 \big( (0,T_0), L^{\infty}(\RN) \big)$ 
and hence $I$ is differentiable with respect to $t$ at $u$ along the smooth direction $u_t$.
By a direct computation and from \ef{eq:1.1}, we have
	\begin{equation*}
	\frac{d}{dt} I \big( u(\cdot ,t) \big)=I'(u(\cdot,t))(u_t(\cdot,t))
	=- \int_{\RN} u_t^2(x,t) \,dx.
	\end{equation*}
This completes the proof.
\end{proof}


\noindent
Lemma \ref{energy-est} implies that $I$ is decreasing in $t$ 
and hence $I$ is a Lyapunov function associated with the problem \ef{eq:1.1}-\ef{eq:1.2}.

\begin{lemma}[Flow stabilization] 
	\label{stab}
	For every $K>0$ we have
	\begin{align*}
&	\lim_{t\to\infty} \sup_{\tau\in [0,K]}
\|u(\cdot, t+\tau)-u(\cdot,t)\|_{L^2(\RN)}=0, \\
& \lim_{ t \to \infty} \sup_{ \tau\in [0,K]} 
\| u(\cdot,t+\tau)-u(\cdot,\tau) \|_{C^1(\RN)} =0.
\end{align*}
In particular if $u( \cdot,t_n)\to w$ uniformly in $\RN$, 
then $u( \cdot,t_n+\rho_n)\to w$ in $C^1(\R^N)$
for any bounded sequence $\{\rho_n\}_{n\in\N}\subset\R^+$.
\end{lemma}
\begin{proof}
	We fix $\tau \in [0,K]$. For every $t>0$ we have
	\begin{align*}
	\int_{\R^N}|u(\cdot,t+\tau)-u(\cdot,t)|^2 \,dx&=
	\int_{\R^N}\left|\int_t^{t+\tau} u_t(\cdot,s) \,ds \right|^2\,dx \leq
	\tau\int_{\R^N}\int_t^{t+\tau}|u_t(\cdot,s)|^2\,ds \,dx \\
	&=\tau \int_t^{t+\tau}\int_{\R^N}|u_t(\cdot,s)|^2 \,dx\,ds
	=\tau \big( I(t)-I(t+\tau) \big).
	\end{align*}
Since $I(t)$ is non-increasing and bounded from below, 
it has finite limit as $t\to\infty$, which yields the assertion.
For the second claim, since $\{u(\cdot,t), \ t>1 \}$ is relatively compact in $C^1(\RN)$ 
(see the argument in the proof of Lemma~\ref{lem:3.1}), 
one can argue as in \cite[Lemma 3.1]{MSS}.
\end{proof}

\noindent
By Lemma \ref{stab}, we have the following basic result.
\begin{lemma}[$\omega$ limit structure]
\label{limite}	
$\Omega(u)$ is either $\{ 0\}$ or consists of positive solutions of \ef{eq:1.3}.
\end{lemma}
\begin{proof}
For every $\varphi \in C_0^{\infty}(\RN)$ and $\tau>0$, we have
\begin{align*}
\int_{t_n}^{t_n+\tau}\int_{\R^N} u_t \varphi \,dx\,ds 
&+\int_{t_n}^{t_n+\tau}\int_{\R^N} \nabla u\cdot\nabla \varphi \,dx\,ds
+2\int_{t_n}^{t_n+\tau}\int_{\R^N} u\varphi|\nabla u|^2 \,dx\,ds \\
&\hspace{-2em} 
+2\int_{t_n}^{t_n+\tau}\int_{\R^N} u^2\nabla u\cdot\nabla \varphi \,dx\,ds
+\int_{t_n}^{t_n+\tau}\int_{\R^N} u\varphi \,dx\,ds
=\int_{t_n}^{t_n+\tau}\int_{\R^N} u^p\varphi \,dx\,ds.  \notag
\end{align*}
This yields for some $\xi_n\in [t_n,t_n+\tau]$
\begin{align*}
	&\int_{\R^N} \big( u(x,t_n+\tau)-u(x,t_n) \big) \varphi(x) \,dx +\int_{\R^N} \nabla u(x,\xi_n)\cdot\nabla \varphi(x) \,dx \\
	&+2\int_{\R^N} u(x,\xi_n)\varphi(x)|\nabla u(x,\xi_n)|^2 \,dx +2\int_{\R^N} u^2(x,\xi_n)\nabla u(x,\xi_n)\cdot\nabla \varphi(x) \,dx\\
	&+\int_{\R^N} u(x,\xi_n)\varphi(x) \,dx=\int_{\R^N} u^p(x,\xi_n)\varphi(x) \,dx.  \notag
\end{align*}
Since $u(\cdot,t_n) \to w$ uniformly in $\RN$, by virtue of Lemma \ref{stab}
it follows that $u(\cdot,t_n+\tau) \to w$ in $L^2(\R^N)$ and $u(\cdot,\xi_n) \to w$ in $C^1(\R^N)$,  which yields
\begin{align*}
  \int_{\R^N} \nabla w\cdot\nabla \varphi \,dx 
+2\int_{\R^N} w\varphi|\nabla w|^2 \,dx +2\int_{\R^N} w^2\nabla w\cdot\nabla \varphi \,dx
+\int_{\R^N} w\varphi \,dx=\int_{\R^N} w^p\varphi \,dx,  \notag
\end{align*}
for every $\varphi\in C^\infty_0(\R^N)$, namely $w$ is a non-negative solution of \ef{eq:1.3}.
%
\end{proof}	

\begin{lemma}[Energy bounds]
	\label{lem:2.1}
The following properties hold.
\begin{itemize}
\item[\rm(i)] There exists $C>0$ such that 
$$
\sup_{t>0} \int_{\RN} \big( (1+2|u(x,t)|^2)|\nabla u(x,t)|^2
+|u(x,t)|^2 \big) \,dx \le C.$$
\item[\rm(ii)] If $w\in \Omega(u)$, then $w\in H^1(\RN)$.
Moreover 
$$
\sup_{w\in \Omega(u)}\big( \| w\|_{H^1(\RN)}+\| w\|_{\LI} \big)<\infty.
$$
\end{itemize}
\end{lemma}

\begin{proof}
We prove (i).\ Since $I$ is decreasing in $t$, it follows that for $t>0$
$$
\frac{1}{2} \int_{\RN}
\Big((1+2u^2)|\nabla u|^2 +u^2\Big) \,dx
-\frac{1}{p+1} \int_{\RN} |u|^{p+1} \,dx\le I(u_0).
$$
Moreover for every $R>0$, one has
\begin{equation*}
\frac{1}{p+1} \int_{\RN} |u|^{p+1} \,dx
\le \frac{1}{p+1} \Big( \sup_{|x| \ge R, \, t>0} |u(x,t)| \Big)^{p-1}
\int_{B^c(0,R)} |u|^2 \,dx
+\int_{B(0,R)} |u|^{p+1} \,dx.
\end{equation*}
Finally from \ef{eq:1.4}, by taking $R$ large enough, we have
$$
\frac{1}{p+1} \Big( \sup_{ |x| \ge R, \, t>0} |u(x,t)| \Big)^{p-1}
 \le \frac{1}{4},
 $$
which yields
$$
\sup_{t>0} \int_{\RN} \Big((1+2|u(x,t)|^2)|\nabla u(x,t)|^2+|u(x,t)|^2\Big) \,dx
\le 4M|B(0,R)|+4I(u_0),$$
where we have set $M= \| u \|_{L^{\infty}( \RN \times [0,\infty))}^{p+1}$.
\vskip3pt
\noindent
(ii) Suppose that 
$$
\| u(\cdot,t_n)-w(\cdot)\|_{L^{\infty}(\RN)} \to 0 \quad\text{for some $t_n \to \infty$}.
$$
Then from (i), it follows that $\{ u(\cdot,t_n) \}$ is bounded in $H^1(\RN)$. 
Thus, up to a subsequence, we have $u(\cdot,t_n) \rightharpoonup \tilde{w}$ in $H^1(\RN)$
and $u(\cdot,t_n) \to \tilde{w}$ a.e.\ in $\RN$ for some $\tilde{w} \in H^1(\RN)$.
Since $u(\cdot,t_n)$ converges to $w$ uniformly, it follows that 
$w \equiv \tilde{w}$, which implies that $w\in H^1(\RN)$.\
By the boundedness of $u(x,t)$ in $H^1(\RN)$ and $L^{\infty}(\RN)$,
the last assertion of (ii) follows.
\end{proof}

\begin{lemma}[Lipschitzianity controls]
	\label{lem:2.2}
Let $0<t_1<t_2$. Then there exists $C>0$ independent of $t_1$ and $t_2$ such that
the following properties hold.
\begin{itemize}
\item[\rm(i)] $ \displaystyle
\| u(\cdot,t_2)-w(\cdot) \|_{\LN(\RN)}
\le e^{C(t_2-t_1)}
\| u(\cdot,t_1)-w(\cdot)\|_{\LN(\RN)}.$
\item[\rm(ii)] 
$ \displaystyle 
\int_{t_1}^{t_2}
\| u(\cdot,s)-w(\cdot) \|_{\HN(\RN)}^2 \,ds
\le Ce^{C(t_2-t_1)}
\| u(\cdot,t_1)-w(\cdot) \|_{\LN(RN)}^2.
$
\end{itemize}
\end{lemma}

\begin{proof}
(i) The proof is based on the standard energy estimate. 
We put $\phi(x,t)=u(x,t)-w(x)$. 
From \ef{eq:1.1}, \ef{eq:1.3} and by the mean value theorem, one has 
\begin{equation} \label{eq:2.8.1}
\phi_t-(1+2u^2)\Delta \phi -2 w \nabla(u+w) \cdot \nabla \phi
-2(u+w) \Delta w \phi -2| \nabla u|^2 \phi +\phi
-p\big( \kappa u+(1-\kappa)w \big)^{p-1} \phi =0,
\end{equation}
for some $\kappa \in (0,1)$.
Multiplying \ef{eq:2.8.1} by $\phi$ and integrating it over $\RN$, we get
\begin{align*}
\frac{1}{2} \frac{\partial }{\partial t} \| \phi\|_{L^2}^2
&- \int_{\RN} \Big( (1+2u^2)\phi \Delta \phi +2w\phi \nabla (u+w) \cdot \nabla \phi
+2(u+w) \phi^2 \Delta w +2 \phi^2 |\nabla u|^2 \Big) \,dx\\
& +\int_{\RN} \phi^2 \,dx -p \int_{\RN} \big( \kappa u+(1-\kappa)w \big)^{p-1} \phi^2 \,dx =0.
\end{align*}
Using the integration by parts, we have
\begin{align*}
-\int_{\RN} (1+2u^2) \phi \Delta \phi \,dx
&= \int_{\RN} (1+2u^2) |\nabla \phi|^2 +4 u \phi \nabla u \cdot \nabla \phi \,dx, \\
-\int_{\RN} 2(u+w)\phi^2 \Delta w \,dx
&= \int_{\RN} 2 \phi^2 \nabla w \cdot \nabla (u+w)
+4(u+w) \phi \nabla w \cdot \nabla \phi \,dx.
\end{align*}
Thus one has
\begin{align*}
\frac{1}{2} \frac{\partial}{\partial t} \| \phi\|_{L^2}^2
&+ \int_{\RN} \Big\{ (1+2u^2)|\nabla \phi|^2 +4 u \phi \nabla u \cdot \nabla \phi
-2w \phi \nabla (u+w) \cdot \nabla \phi \\
&\hspace{4em} +2 \phi^2 \nabla w \cdot \nabla (u+w) 
+4(u+w) \phi \nabla w \cdot \nabla \phi
-2\phi^2 |\nabla u|^2 \Big\} \,dx \\
&+\int_{\RN} \phi^2 \,dx
-p \int_{\RN} \big( \kappa u+(1-\kappa)w \big)^{p-1} \phi^2 \,dx =0.
\end{align*}
Since $u$, $w$, $\nabla u$ and $\nabla w$ are bounded, we obtain
$$
\frac{1}{2}\frac{\partial }{\partial t} \| \phi\|_{L^2}^2 +\| \phi\|_{H^1}^2 
\le C \| \phi\|_{L^2} \| \nabla \phi\|_{L^2} +C \| \phi \|_{L^2}^2.$$
Thus by the Young inequality, it follows that
\begin{equation}\label{eq:2.2}
\frac{\partial}{\partial t} \| \phi(\cdot,t)\|_{L^2}^2
+\| \phi(\cdot,t)\|_{H^1}^2 \le C \| \phi(\cdot,t) \|_{L^2}^2.
\end{equation}
Now let $\zeta(t):=\| \phi(\cdot,t) \|_{L^2}^2$.
Then one has $\zeta'(t) \le C \zeta(t)$. 
By the Gronwall inequality, it follows that
$\zeta(t_2) \le e^{C(t_2-t_1)} \zeta(t_1)$ and hence the claim holds.

(ii) Integrating \ef{eq:2.2} over $[t_1,t_2]$, one has
\begin{align*}
\int_{t_1}^{t_2} \| \phi(\cdot,s) \|_{H^1}^2 \,ds
\le \| \phi(\cdot,t_1) \|_{L^2}^2
+C \int_{t_1}^{t_2} \| \phi(\cdot,s) \|_{L^2}^2 \,ds.
\end{align*}
Thus from (i), we get
$$
\int_{t_1}^{t_2} \| \phi(\cdot,s) \|_{H^1}^2 \,ds
\le \Big( 1+\frac{1}{2C} \Big)e^{2C(t_2-t_1)}
\| \phi(\cdot,s) \|_{L^2}^2.$$
This completes the proof.
\end{proof}

\begin{lemma}[Further stability estimates]
	\label{lem:2.3}
Let $K>1$ be arbitrarily given and $\{ t_n\}_{n\in\N}$ be a sequence such that 
$t_n \to \infty$ as $n \to \infty$. 
If $\| u(\cdot,t_n)-w(\cdot) \|_{\LI} \to 0$ 
as $n \to \infty$, then the following properties hold.
\begin{itemize}
\item[\rm(i)] 
$\displaystyle \lim_{n \to \infty} 
\int_0^K \| u(\cdot,s+t_n)-w(\cdot) \|_{\HN(\RN)}^2 \,ds =0$.
\item[\rm(ii)]
$\displaystyle 
\lim_{ n \to \infty} \| u(\cdot,t+t_n)-w(\cdot) \|_{L^{\infty}(\RN \times [0,K])}=0.$
\end{itemize}
\end{lemma}

\begin{proof}
Arguing as in the proof of Lemma \ref{lem:2.1} (ii), 
we may assume that $u(\cdot,t_n) \rightharpoonup w$ in $H^1(\RN)$.
Moreover by the uniform decay condition \ef{eq:1.4}, one can show that
$\sup_{n \geq 1} u(x,t_n)$ decays exponentially at infinity, see Lemma \ref{lem:6.4}. 
Thus by the exponential decay of $w$ and the embedding 
$H^1(\RN) \hookrightarrow L^2_{{\rm loc}}(\RN)$, it follows that
\begin{equation}\label{eq:2.3}
\lim_{ n \to \infty} \| u(\cdot,t_n)-w(\cdot) \|_{L^2(\RN)}=0.
\end{equation}
Next applying Lemma \ref{lem:2.2} (ii) with $t_1=t_n$ and $t_2=t_n+K$, one has
$$
\int_{t_n}^{t_n+K} \| u(\cdot,s)-w(\cdot) \|_{H^1}^2 \,ds
\le Ce^{CK} \| u(\cdot,t_n)-w(\cdot) \|_{L^2}^2.$$
Thus from \ef{eq:2.3}, the claim holds.

(ii) 
We argue as in \cite[Theorem 2.5, P. 18]{LSU}.
Let $\phi_n(x,t)=u(x,t+t_n)-w(x)$ and define
\begin{align*}
\tilde{L}(\phi_n)&=
(1+2u^2) \Delta \phi_n +2w \nabla (u+w) \cdot \nabla \phi_n+a(x) \phi_n -( \phi_n)_t,\\
a(x) &= 2(u+w) \Delta w +2 |\nabla u|^2-1+p \big( \kappa u+(1-\kappa)w \big)^{p-1}.
\end{align*}
Then from \ef{eq:2.8.1}, it follows that $\tilde{L}(\phi_n)=0$.
We put
\begin{equation} \label{eq:2.9.1}
\| u\|_{L^{\infty} \big( \RN \times [0,K] \big)}
+\| \nabla u\|_{L^{\infty} \big( \RN \times [0,K] \big)} = M, \
\| a\|_{L^{\infty}(\RN)} = A, \ 
\| \phi_n(\cdot ,0) \|_{L^{\infty}(\RN)} = B.
\end{equation}
For $\varepsilon>0$, $R>0$ and $c>0$, we define
$$
Z(x,t)=\phi_n(x,t)e^{-(A+\varepsilon)t} -B-\frac{M}{R^2} (x^2+ct), \ 
|x|\le R, \ t\in [0,K].
$$
Then by a direct calculation, one has
\begin{align*}
(\tilde{L}-A-\varepsilon) Z &= B \big( A+\varepsilon-a(x) \big) \\
&\quad +\frac{M}{R^2} \Big( c + \big( A+\varepsilon-a(x) \big)(x^2+ct)
-2(1+2u^2)-4w\nabla (u+w) \cdot \nabla x \Big) \\
&=: F(x,t).
\end{align*}
From \ef{eq:2.9.1} and the boundedness of $u$, $w$, $\nabla u$, $\nabla w$, 
we can choose large $c$ independent of $\varepsilon$, $R$ and $n\in \N$ so that $F \ge 0$.
Moreover from \ef{eq:2.9.1}, we also have
\begin{align*}
Z(x,t) &= \phi_ne^{-(A+\varepsilon)t}-B-M-\frac{M}{R^2}ct \le 0 \quad 
\hbox{on} \ |x|=R, \ t\in [0,K],\\
Z(x,0) &= \phi_n(x,0)-B-\frac{M}{R^2}x^2 \le 0 \quad 
\hbox{for} \ |x| \le R.
\end{align*}
Thus by applying the Comparison Principle to $\tilde{L}-A-\varepsilon$, 
we obtain $Z \le 0$ for $|x| \le R$ and $t\in [0,K]$. 
Defining
$$
z(x,t)=\phi_n(x,t) e^{-(A+\varepsilon)t}+B+\frac{M}{R^2}(x^2+ct)$$
for same $c>0$, one can see that
$$
(\tilde{L}-A-\varepsilon)z=f \le 0, \quad 
z \ge 0 \ \hbox{on} \ |x|=R, \ t\in [0,K], \quad 
z(x,0) \ge 0 \ \hbox{for} \ |x|\le R.
$$
Thus by the Comparison Principle, we get $z \ge 0$ and hence
$$
| \phi_n(x,t)| \le e^{(A+\varepsilon)t}\left( B+ \frac{M}{R^2}(x^2+ct) \right)
\quad \hbox{for} \ |x|\le R, \ t\in [0,K].$$
Since $c$ is independent of $\varepsilon$ and $R$, 
we can take $R \to \infty$, $\varepsilon \to 0$ to obtain 
$$
\| \phi_n(\cdot,t) \|_{L^{\infty} \big( \RN \times [0,K] \big)}
\le Be^{AK} = e^{AK} \| \phi_n(\cdot ,0) \|_{L^{\infty}(\RN)}.$$
Then by the assumption $\| \phi_n(\cdot,0)\|_{L^{\infty}}=
\| u(\cdot,t_n)-w(\cdot)\|_{L^{\infty}}\to 0$, it follows that
$$ \| \phi_n( \cdot,t)\|_{L^{\infty} \big( \RN \times [0,K] \big)} \to 0
\quad \hbox{as} \ n \to \infty.$$
This completes the proof.
\end{proof}

\begin{lemma}[Further stability estimates]
	\label{lem:2.4}
Let $K>1$.\ Then there exists $C=C(K)>0$ with
$$
\int_0^K \| u(\cdot,s+\tau +t)-w(\cdot)\|_{\HN(\RN)}^2 \,ds
\le C \int_0^K \| u(\cdot,s+t)-w(\cdot) \|_{\HN(\RN)}^2 \,ds
$$
for any $t>0$ and $\tau \in [0,K]$.
\end{lemma}

\begin{proof}
Although the proof proceeds as in  \cite[Proposition 4.2]{CGH},
we will sketch it for the sake of completeness.\
By the mean value theorem and Schwarz inequality, 
there is $s_0\in [0,K]$ with
\begin{align*}
\| u(\cdot,t+s_0)-w(\cdot) \|_{L^2}
&\le \| u(\cdot,t+s_0)-w(\cdot) \|_{H^1} \\
&\le C\int_t^{t+K} \| u(\cdot,s)-w(\cdot) \|_{H^1} \,ds \\
&\le C \Big( \int_t^{t+K} 
\| u(\cdot,s)-w(\cdot) \|_{H^1}^2 \,ds \Big)^{1\over 2}.
\end{align*}
Thus by applying Lemma \ref{lem:2.2} (ii) with $t_1=t+s_0$ and
$t_2=t+s_0+2K$, it follows that
\begin{align*}
\int_{t+s_0}^{t+s_0+2K}
\| u(\cdot,s)-w(\cdot)\|_{H^1}^2 \,ds
&\le Ce^{CK} \| u(\cdot,t+s_0)-w(\cdot) \|_{L^2}^2 \\
&\le Ce^{CK} \int_t^{t+K} \| u(\cdot,s)-w(\cdot) \|_{H^1}^2 \,ds.
\end{align*}
Let $\tau \in [0,K]$. Since $[t+\tau,t+\tau+K] \subset
[t,t+2K] \subset [t,t+K] \cup [t+s_0,t+s_0+2K]$, we obtain
\begin{align*}
\int_{t+\tau}^{t+\tau+K}
\| u(\cdot,s)-w(\cdot) \|_{H^1}^2 \,ds 
\le (1+Ce^{CK})\int_t^{t+K} \| u(\cdot,s)-w(\cdot) \|_{H^1}^2 \,ds.
\end{align*}
This completes the proof.
\end{proof}

\begin{lemma}\label{lem:2.5}
Let $\{z_n\}_{n\in\N} \subset \RN$ be a sequence with $| z_n |\le 1$ for all $n\in\N$.\
Then there exists $C>0$ independent of $n \in \N$ such that
the following properties hold.
\begin{itemize}
\item[\rm(i)]
$\| w(\cdot+z_n)-w(\cdot)-\nabla w(\cdot) \cdot z_n \|_{\HN(\RN)}\le C|z_n|^2$.
\item[\rm(ii)]
$\displaystyle \| w(\cdot+z_n)-w(\cdot) \|_{\HN(\RN)} \le C|z_n|$.
\end{itemize}
\end{lemma}

\begin{proof}
(i) By the Taylor expansion, one has
$$
| w(x+z_n)-w(x)-\nabla w(x) \cdot z_n|
\le C |z_n|^2 \sum_{i,j=1}^N \Big| \frac{\partial^2 w}{\partial x_i \partial x_j}(x+\kappa_n z_n) \Big|,
$$
for some $\kappa_n \in (0,1)$.
From \ef{eq:1.3} and by the exponential decay of $w$, we can show that
$\frac{\partial^3 w}{\partial x_i \partial x_j \partial x_k}$
also decays exponentially at infinity for all $i,j,k=1,\cdots,N$. Thus, we get
\begin{align*}
\| w(\cdot+z_n)-w(\cdot)-\nabla w(\cdot)\cdot z_n \|_{H^1}
\le C |z_n|^2 \sum_{i,j=1}^N 
\Big\| \frac{\partial^2 w}{\partial x_i \partial x_j} (\cdot+\kappa_n z_n)
\Big\|_{H^1}\le C|z_n|^2.
\end{align*}
(ii) We differentiate \ef{eq:1.3} w.r.t\  to $x_i$.
Then multiplying by $\frac{\partial w}{\partial x_i}$ and integrating on $\RN$, yields
\begin{align*}
&\int_{\RN} \Big\{ (1+2w^2) \left| \nabla \frac{\partial w}{\partial x_i} \right|^2 +2 \left( \frac{\partial w}{\partial x_i} \right)^2 |\nabla w|^2\\
&\qquad \quad +8w \frac{\partial w}{\partial x_i} \nabla w \cdot \nabla \frac{\partial w}{\partial x_i} +\left( \frac{\partial w}{\partial x_i} \right)^2 
-pw^{p-1} \left( \frac{\partial w}{\partial x_i} \right)^2 \Big\} \,dx=0.
\end{align*}
Then by the Schwarz inequality, Young inequality and from the boundedness of $w$, $\nabla w$, 
we get
\begin{equation}\label{eq:2.4}
\int_{\RN} \left|\nabla \frac{\partial w}{\partial x_i} \right|^2 \,dx
\le C \int_{\RN} \left| \frac{\partial w}{\partial x_i} \right|^2 \,dx
\le C \| w\|_{H^1}^2.
\end{equation}
Thus from (i), \ef{eq:2.4} and $|z_n|\le 1$, it follows that
\begin{align*}
\| w(\cdot+z_n)-w(\cdot)\|_{H^1}
&\le \| w(\cdot+z_n)-w(\cdot)-\nabla w(\cdot)\cdot z_n \|_{H^1}
+\| \nabla w(\cdot)\cdot z_n \|_{H^1} \\
&\le C|z_n|^2+C|z_n| \le C|z_n|.
\end{align*}
This completes the proof.
\end{proof}

\subsection{Decay estimates}
In this subsection, we show uniform estimates for global solutions of \ef{eq:1.1}-\ef{eq:1.2}. 
Our goal of this subsection is to prove the following proposition.

\begin{proposition}[Uniform decay]
	\label{prop:6.1}
Let $u(x,t)$ be a non-negative, radially non-increasing and globally defined solution of \ef{eq:1.1}-\ef{eq:1.2}. 
Then the following properties hold.
\begin{itemize}
\item[\rm(i)] $\displaystyle \sup_{t>0} \| u(\cdot,t) \|_{L^{\infty}(\RN)} <\infty$.
\item[\rm (ii)] $\displaystyle \lim_{|x|\to \infty} \sup_{t>0} u(x,t)=0$.
\end{itemize}
\end{proposition}

\noindent
The proof of Proposition \ref{prop:6.1} consists of several lemmas. 
First, we prove that $u$ is uniformly bounded near infinity.

\begin{lemma}[Universal bound near infinity]
	\label{lem:6.2}
Let $u(x,t)$ be a non-negative globally defined solution of \ef{eq:1.1}-\ef{eq:1.2} and assume that $u$ is radially non-increasing with respect to the origin. Then for any 
$$
K> \Big( \frac{p+1}{2} \Big)^{1 \over p-1},
$$
there exists $R_K>0$ such that
\begin{equation}\label{eq:6.1}
u(x,t) \le K \quad \hbox{for all} \ |x| \ge R_K \ \hbox{and} \ t>0.
\end{equation}
\end{lemma}

\begin{proof}
Suppose by contradiction that the claim fails. 
Then we find $K_0>((p+1)/2)^{1/(p-1)}$ such that for all $R>0$, 
$u(x_R, t_R) >K_0$ for some $| x_R| \ge R$ and $t_R>0$. 
For simplicity, we write $|x_R|= \tilde{R}$. 
Since $u$ is radially non-increasing, it follows that
\begin{equation}\label{eq:6.2}
u(x,t_R) >K_0 \quad \hbox{for all $x \in B(0,\tilde{R})$}.
\end{equation}
We claim that $u(x,t)$ must blow-up in finite time. We define a functional $I_R$ by
\begin{equation}\label{eq:6.3}
I_R(u):= \frac{1}{2} \int_{B(0,R)} \big((1+2u^2)|\nabla u|^2 +u^2\big) \,dx
	-\frac{1}{p+1} \int_{B(0,R)} |u|^{p+1} \,dx.
\end{equation}
First, for sufficiently large $R>1$, 
we show that there exists a function $v_R \in C_0^{\infty}(\RN)$ such that 
\begin{equation}\label{eq:6.4}
I_R(v_R)<0, \quad 
v_R \le K_0 \quad \hbox{in} \ B(0,R), \quad 
v_R=0 \quad \hbox{on} \ \partial B(0,R).
\end{equation}
To this aim, let 
\begin{equation}
	\label{gamma-cond}
\Big( \frac{p+1}{2}\Big)^{1\over p-1}<\zeta<K_0
\end{equation} 
be arbitrarily given and choose $v_R \in C_0^{\infty}(\RN)$ so that
$0 \le v_R(x) \le \zeta$ for all $x\in \RN$,
$$
v_R(x)= \zeta \quad \hbox{for} \ |x| \le R-1, \quad v_R(x)=0 \quad \hbox{for} \ |x| \ge R,
\quad  \ | \nabla v_R(x)| \le C\zeta.
$$
Then we have
\begin{align*}
I_R(v_R) & \le \int_{ \{R-1 \le |x| \le R\}} \Big(
	\frac{C^2\zeta^2}{2}(1+2v_R^2)+\frac{v_R^2}{2}-\frac{v_R^{p+1}}{p+1} \Big) \,dx
	+|B(0,R-1)| \Big( \frac{\zeta^2}{2}-\frac{\zeta^{p+1}}{p+1} \Big) \\
	&\le \big( R^N-(R-1)^N \big) C^2|B(0,1)| \zeta^2(1+\zeta^2)
	+(R-1)^N |B(0,1)| \Big( \frac{\zeta^2}{2}-\frac{\zeta^{p+1}}{p+1} \Big).
\end{align*}
By \eqref{gamma-cond}, it follows that $\zeta^2/2-\zeta^{p+1}/(p+1)<0$ and hence
$I_R(v_R) \to - \infty$, as $R \to \infty$. 
Thus, by taking large $R>1$, we obtain $I_R(v_R)<0$. 
Moreover, by the construction, we also have $v_R \le K_0$ in $B(0,R)$ and 
$v_R=0$ on $\partial B(0,R)$. Finally since $\tilde{R} \ge R$, 
we can replace $R$ by $\tilde{R}$. 
Next we consider the following auxiliary problem:
\begin{align}\label{eq:6.5}
	v_t=(1+2v^2)\Delta v+2v|\nabla v|^2-v+v^p 
	&\quad\,\, \hbox{in} \ B(0,\tilde{R}) \times [t_R,\infty), \\
	v=0 &\quad\,\, \hbox{on} \ \partial B(0,\tilde{R}) \times [t_R,\infty), \nonumber \\
	v(x,t_R)=v_{\tilde{R}}(x) &\quad\,\, \hbox{in} \ B(0,\tilde{R}). \nonumber
\end{align}
We claim that $v(x,t)$ blows up in finite time in a similar argument as \cite{GK2}. 
Indeed by a direct calculation, 
$\frac{d}{dt} I_{\tilde{R}} \big( v(\cdot ,t) \big) \le 0$ for $t \ge t_R$. 
Next from \ef{eq:6.5}, we obtain
\begin{align*}
	\frac{\partial}{\partial t}\Bigl( \frac{1}{2} \int_{B(0,\tilde{R})} v^2(x,t) \,dx \Bigr)
	&= - \int_{B(0,\tilde{R})} \Big((1+4v^2)|\nabla v|^2 +v^2 -v^{p+1}\Big) \,dx\\
	&= -4I_{\tilde{R}} \big( v(\cdot,t) \big)
	+\Big(1-\frac{4}{p+1} \Big) \int_{B(0,\tilde{R})} v^{p+1} \,dx\\
	&\ge -4I_{\tilde{R}}(v_{\tilde{R}}) +\frac{p-3}{p+1}
	\int_{B(0,\tilde{R})} v^{p+1} \,dx.
\end{align*}
By the H\"older inequality, we also have
$$
|B(0, \tilde{R})|^{- \frac{p-1}{2}} \Big( \int_{B(0,\tilde{R})} v^2 \,dx \Big)^{ p+1 \over 2} 
\le \int_{B(0,\tilde{R})} v^{p+1} \,dx. 
$$
Thus from $I_{\tilde{R}}(v_{\tilde{R}})<0$ and $p\geq 3$, we obtain
$$
\frac{\partial }{\partial t} \Big( \int_{B(0,\tilde{R})} v^2(x,t) \,dx \Big) 
\ge C \Big( \int_{B(0,\tilde{R})} v^2(x,t) \,dx \Big)^{p+1 \over 2} 
\quad \hbox{for all} \ t \ge t_R.
$$
This implies that $v(x,t)$ blows up in finite time. 
Now from \ef{eq:6.2} and \ef{eq:6.4}, one has
$$
L(v) \ge 0 \quad \hbox{in} \ B(0,\tilde{R}) \times [t_R,\infty), \quad 
v \le u \quad \hbox{on} \ \partial B(0,\tilde{R}) \times [t_R,\infty), \quad 
v(\cdot,t_R) \le u(\cdot,t_R) \ \hbox{in} \ B(0,\tilde{R}). 
$$
Then by Lemma \ref{lem:0.2}, it follows that 
$u(x,t) \ge v(x,t)$ for all $x\in B(0,\tilde{R})$ and $t \ge t_R$. 
Thus $u(x,t)$ must blow-up in finite time, 
contradicting the assumption that $u(x,t)$ is globally defined. 
\end{proof}

\begin{remark}\label{rem:6.3}\rm
Lemma \ref{lem:6.2} is the only part where the radial non-increase of $u(x,t)$ is needed. 
We can remove this assumption if we could show that
$$
\max_{x \in \partial B(0,R+2)} u(x,t) \le \inf_{x \in B(0,R)} u(x,t) 
\quad \hbox{for} \ t\in [0,T] \ \hbox{and large} \ R>0. 
$$
This type of estimates were obtained for porous medium equations, see \cite[Proposition 2.1]{AC}. 
However we don't know whether this estimate holds true for our quasi-linear parabolic problem.
\end{remark}

\noindent
Once we have the uniform boundedness near infinity, 
we can get the decay estimate at infinity. 

\begin{lemma}[Exponential decays]
	\label{lem:6.4}
Suppose $N \ge 2$ and let $u(x,t)$ be a non-negative global solution of \ef{eq:1.1}-\ef{eq:1.2}
which satisfies the uniform boundedness property \ef{eq:6.1}. 
Then there exist $\delta>0$, $C>0$ and $R_0>0$ such that
$$
\sup_{t>0} | D^k u(x,t) | \le Ce^{-\delta |x|} \quad 
\hbox{for all} \ |x| \ge R_0 \ \hbox{and} \ |k|\le 2.
$$
\end{lemma}

\begin{proof}
By standard linear parabolic estimates, it suffices to consider the case $k=0$. 
Let $w$ be a positive solution of \ef{eq:1.3}.\ Then $w$ is radially decreasing 
and decays exponentially at infinity. 
Moreover we claim that $w(0)> ((p+1)/2)^{1/(p-1)}$. 
Indeed, $w$ satisfies the Poh\v ozaev identity
$$
0 \le \frac{N-2}{2N} \int_{\RN} (1+2w^2) |\nabla w|^2 \,dx 
= \int_{\RN} \Big( \frac{w^{p+1}}{p+1}-\frac{w^2}{2} \Big) \,dx. 
$$
For the proof, see \cite[Lemma 3.1]{CJS}. 
If the claim fails, 
then $w(x) < ((p+1)/2)^{1/(p-1)}$ for all $x\in \RN \setminus \{ 0\}$ by the monotonicity of $w$, 
which implies
$$
\int_{\RN} \left( \frac{w^{p+1}}{p+1}-\frac{w^2}{2} \right) \,dx < 0, 
$$
which is impossible. Now applying Lemma \ref{lem:6.2} with 
$$
\Big( \frac{p+1}{2} \Big)^{1\over p-1}<K<w(0), 
$$ 
there exists $R_0=R(K)$ such that
\begin{equation} \label{eq:6.6}
u(x,t) \le K \ \hbox{for} \ |x| \ge R_0 \ \hbox{and} \ t>0. 
\end{equation}
Moreover choosing $R_0$ larger if necessary, we may assume 
${\rm supp}(u_0) \subset B(0,R_0)$. Next, we put 
$$
Z(x,t):=Z(x)=w(|x|-R_0),\quad  \text{for $|x| \ge R_0$ and $t>0$}. 
$$
Then there exists $\varepsilon_0>0$ such that $Z(x) \ge K$ for $R_0 \le |x| \le R_0+\varepsilon_0$. From \ef{eq:6.6}, we get
\begin{align*}
L(Z) \le 0 &\quad \hbox{for} \ \ R_0\le |x| \le R_0+\varepsilon_0 \ \hbox{and} \ t>0, \\ 
Z \ge u &\quad \hbox{for} \ \ |x|=R_0, R_0+\varepsilon_0 \ \hbox{and} \ t>0, \\ 
Z(\cdot,0) \ge u_0 &\quad \hbox{for} \ \ R_0 \le |x| \le R_0+\varepsilon_0. 
\end{align*}
Thus by Lemma \ref{lem:0.2}, 
we obtain $u(x,t) \le U(x)$ for $R_0 \le |x| \le R_0+\varepsilon_0$ and $t>0$. 
Applying the Comparison Principle again, we have 
$u(x,t) \le U(x)$ for all $|x| \ge R_0$ and $t>0$. This completes the proof. 
\end{proof}

\begin{remark}\label{rem:6.5}\rm
In the proof of Lemma \ref{lem:6.4}, our construction of a supersolution $Z$ fails when $N=1$. 
In fact in this case, we claim that 
$$
w(0)= \Big( \frac{p+1}{2} \Big)^{1\over p-1}. 
$$
To see this, we multiply $w'$ by the one-dimensional version of equation \ef{eq:1.3} 
$$
(1+2w^2)w''+2w(w')^2 -w +w^p=0. 
$$
Integrating it over $[0,r]$, since $w'(0)=0$, we have 
$$
\frac{1}{2}\big( 1+2w^2(r) \big) \big( w'(r) \big)^2 +\frac{w^{p+1}(r)}{p+1}-\frac{w^2(r)}{2} 
= \frac{w^{p+1}(0)}{p+1}-\frac{w^2(0)}{2} \ \hbox{for} \ r>0. 
$$
Passing to a limit $r \to \infty$, the claim is proved. 
Since there is no gap between $w(0)$ and $((p+1)/2)^{1/(p-1)}$, 
we cannot apply Lemma \ref{lem:6.2} for $N=1$. 
But if we could replace $((p+1)/2)^{1/(p-1)}$ by 1 in Lemma \ref{lem:6.2}, 
we could construct a decaying supersolution $Z$ in the same way. 
More precisely, instead of \ef{eq:6.1}, let us assume that for any $K>1$, 
there exists $R_K>0$ such that 
$$
u(x,t) \le K, \quad \hbox{for all} \ |x| \ge R_K \ \hbox{and} \ t>0. 
$$
Then the same conclusion as Lemma \ref{lem:6.4} holds. 
On the other hand, replacing $((p+1)/2)^{1/(p-1)}$ by 1, 
our construction of a blow up subsolution $v$ in the proof of Lemma \ref{lem:6.2} fails. 
Thus we need another argument when $N=1$. 
We also remark that a construction of blowing up subsolutions for semi-linear problems as in \cite{delpino-cort} does not work for our problem.
\end{remark}

\begin{lemma}[Global existence and energy sign I]
	\label{lem:6.6}
Let $u$ be a global solution of \ef{eq:1.1}-\ef{eq:1.2}.
Then $I \big( u(\cdot,t) \big) \ge 0$ for every $t>0$. 
\end{lemma}

\begin{proof}
We use the concavity method as in \cite{Le}. 
It suffices to show that if $I \big( u(\cdot,t_0) \big) <0$ for some $t_0>0$, 
then $u(x,t)$ must blow-up in finite time.
To this end, suppose by contradiction that $I \big( u(\cdot,t_0) \big) <0$ but
$u$ is globally defined.
First multiplying \ef{eq:1.1} by $u$ and integrating it over $\RN$, one has 
$$
\int_{\RN} uu_t \,dx 
+\int_{\RN} \big((1+4u^2)|\nabla u|^2+u^2 \big)dx-\int_{\RN} |u|^{p+1} \,dx =0. 
$$
Thus by the definition of $I(u)$, it follows that 
$$
\frac{p-1}{2} \int_{\RN} \big(|\nabla u|^2+u^2\big) dx +(p-3) \int_{\RN} u^2|\nabla u|^2 \,dx 
=(p+1)I(u)+\int_{\RN} uu_t \,dx. 
$$
We put 
$$
M(t):= \frac{1}{2} \int_{t_0}^t \| u(\cdot,s) \|_{L^2(\RN)}^2 \,ds.$$
Then one has $M'(t)=\frac{1}{2} \| u(\cdot,t) \|_{L^2}^2$.
Moreover by Lemma \ref{energy-est} and from $p \ge 3$, we also have
\begin{align*}
M''(t) &= \int_{\RN} uu_t \,dx \\
&= -(p+1) I \big( u(\cdot,t) \big) 
+\frac{p-1}{2} \int_{\RN} \big( \nabla u|^2 +u^2 \big) \,dx
+(p-3) \int_{\RN} u^2 |\nabla u|^2 \,dx \\
&\ge -(p+1) I \big( u(\cdot,t_0) \big) >0 \quad \hbox{for} \ t \ge t_0.
\end{align*}
This implies that $M'(t) \to \infty$ and $M(t) \to \infty$ as $t \to \infty$.
Next by Lemma \ref{energy-est}, it follows that
$$
\int_{t_0}^t \| u_t(\cdot,s) \|_{L^2}^2 \,ds
=I \big( u(\cdot,t_0) \big)-I \big( u(\cdot,t) \big)
<-I \big( u(\cdot,t) \big),$$
which implies that
$$
M''(t) \ge -(p+1) I \big( u(\cdot,t) \big)
> (p+1) \int_{t_0}^t \| u_t(\cdot,s) \|_{L^2}^2 \,ds.$$
Thus we get
\begin{align*}
M(t)M''(t) & \ge \frac{p+1}{2} \left( \int_{t_0}^t \| u(\cdot,s) \|_{L^2}^2 \,ds \right)
\left( \int_{t_0}^t \| u_t(\cdot,s) \|_{L^2}^2 \,ds \right) \\
&\ge \frac{p+1}{2} \left( \int_{t_0}^t u(x,s)u_t(x,s) \,dx\,ds \right)^2 \\
&= \frac{p+1}{2} \big( M'(t)-M'(t_0) \big)^2.
\end{align*}
Since $M'(t) \to \infty$ as $t \to \infty$, there exists $\alpha>0$ and $t_1 \ge t_0$
such that  
$$
M(t)M''(t) \ge (1+\alpha) M'(t)^2 \quad \ \hbox{for} \ t \ge t_1.$$
This shows that $M^{-\alpha}(t)$ is concave on $[t_1,\infty)$,
contradicting to $M^{-\alpha}(t) \to 0$ as $t \to \infty$.
Thus, the assertion holds. 
\end{proof}

\noindent
Finally we show the following lemma.

\begin{lemma}[Global existence and energy sign II]
	\label{lem:6.7}
Let $0<T \le \infty$. 
Let $u$ be a non-negative solution of \ef{eq:1.1}-\ef{eq:1.2} 
and assume that $I \big( u(\cdot,t) \big) \ge 0$ for all $t\in (0,L)$.
Then then there exists $C>0$ such that 
$$
\sup_{t\in (0,L)} \| u(\cdot,t) \|_{L^{\infty}(\RN)} \le C.
$$
\end{lemma}

\begin{proof}
Suppose that there exists a sequence $\{t_n\}_{n\in\N} \subset (0,L)$ 
converging to $L$ such that 
$$
M_n:= \| u(\cdot,t_n) \|_{L^{\infty}(\R^N)} \to \infty. 
$$
We derive a contradiction by using a blow-up type argument. 
Let $\{x_n\}_{n\in\N} \subset \RN$ be such that
$$
\frac{M_n}{2} \le u(x_n,t_n) \le M_n, 
$$
and consider the sequence 
$$
v_n(y,\tau):= 
\frac{1}{M_n} u \left( x_n+\frac{y}{M_n^{p-3 \over 2}}, t_n+\frac{\tau}{M_n^{p-1}} \right). 
$$
Then by a direct calculation, one has 
$$
\frac{\partial v_n}{\partial \tau}
=\frac{1}{M_n^2} \Delta v_n +v_n \Delta v_n^2-\frac{1}{M_n^{p-1}}v_n
 +v_n^p \quad \hbox{in} \ \RN \times ( -M_n^{p-1} t_n ,0]. 
$$
Passing to a subsequence and using a diagonal argument as in \cite{GK1}, we have 
$$
v_n \to v \quad \hbox{in} \ C^{2,1}_{{\rm loc}}(\RN\times (-\infty,0]), 
$$ 
where $v$ is a non-negative solution of the following parabolic problem 
$$
\frac{\partial v}{\partial \tau} =v\Delta v^2 +v^p \quad \hbox{in} \ \RN\times (-\infty,0]. 
$$
Now we claim that $v_{\tau} \equiv 0$. 
To this end, we observe that by Lemma \ref{energy-est} that 
$$
\int_{0}^{t_0} \int_{\RN} | u_t(x,t)|^2 \,dx\,dt 
\le I\big( u(\cdot,0) \big)-I\big( u(\cdot,t_0)\big) \quad \hbox{for any} \ t_0>0. 
$$
Since $I\big( u(\cdot,t) \big) \ge 0$ for all $t>0$, we have
\begin{align*}
	\int_{-\infty}^0 \int_{\R^N} \left| \frac{\partial }{\partial \tau} v_n(y,\tau) \right|^2 \,dy \,d\tau 
	&= \frac{1}{M_n^{2p}} \int_{-\infty}^0 \int_{\R^N} \left| u_t \left( x_n+\frac{y}{M_n^{p-3 \over 2}}, t_n+\frac{\tau}{M_n^{p-1}} \right) \right|^2 \,dy\,d\tau \\ 
	&\le M_n^{ \frac{(N-2)p-3N-2}{2}} \int_0^{\infty} \int_{\RN} | u_t(x,t)|^2 \,dx \,dt \\ 
	&\le  M_n^{ \frac{(N-2)p-3N-2}{2}} I \big( u(\cdot,0) \big). 
\end{align*}
Since $p< (3N+2)/(N-2)$ and $M_n \to \infty$, it follows that $v_{\tau} \equiv 0$. 
Now since $v_{\tau} \equiv 0$, 
$v$ is a nontrivial, non-negative bounded solution of the following nonlinear elliptic problem 
\begin{equation} \label{eq:6.8}
-v\Delta v^2 = v^p \quad \hbox{in} \ \RN. 
\end{equation}
If $3 < p< (3N+2)/(N-2)$, it follows that $v \equiv 0$ by applying the Liouville theorem to $v^2$. 
This contradicts to the fact $v(0) \ge 1/2$. 
On the other hand if $p=3$, it follows that $v^2$ is a nontrivial bounded eigenfunction of 
$-\Delta$ in $\RN$ associated with the eigenvalue $1$. But this is impossible. 
Thus in both cases, we obtain a contradiction and hence the proof is complete. 
\end{proof}

\begin{remark}\label{rem:6.8}\rm
We note that in the proof of Lemma \ref{lem:6.7}, 
we need the assumption $3\le p$ to obtain the non-existence of nontrivial, non-negative bounded solutions of \ef{eq:6.8}. 
We also observe that if we adopt the scaling 
$$
v_n(y,\tau):= 
\frac{1}{M_n} u \Big( x_n+\frac{y}{M_n^{p-1 \over 2}}, t_n+\frac{\tau}{M_n^{p-1}} \Big) 
$$
as in \cite{delpino-cort}, then we obtain the following rescaled problem 
$$
\frac{\partial v_n}{\partial \tau} 
=\Delta v_n +M_n^2 v_n \Delta v_n^2 -\frac{1}{M_n^{p-1}}v_n +v_n^p 
\ \hbox{in} \ \RN \times ( -M_n^{p-1} t_n ,0]. 
$$
Hence this scaling does not work in our case due to the term $M_n^2 v_n \Delta  v_n^2$. 
\end{remark}

\noindent
Now we can see that Proposition \ref{prop:6.1} follows from Lemmas \ref{lem:0.3}, \ref{lem:6.4}, \ref{lem:6.6} and \ref{lem:6.7}.

\subsection{Some technical results}
In this subsection, we prepare some technical lemmas to prove Theorem \ref{thm:1.1}. 
First we shall need the following result. 

\begin{lemma}[$\omega$-limit]
	\label{lem:3.1}
Let $u$ be a non-negative, bounded and globally defined 
solution of \ef{eq:1.1}-\ef{eq:1.2} satisfying the uniform decay condition \ef{eq:1.4}. 
Then $\{ u(\cdot,t_n)\}$ has a uniformly convergent subsequence in $\RN$ 
for any sequence $\{t_n\}_{n\in\N}$ with $t_n \to \infty$. 
In particular, the $\omega$-limit set $\Omega(u)$ is well-defined. 
Furthermore, the set $\{ u(\cdot,t_n)\}$ is relatively compact in $C^1(\RN)$. 
\end{lemma}

\begin{proof}
We know that $\| u(\cdot,t)\|_{L^{\infty}(\RN)}$ is uniformly bounded. 
Moreover by assumption \ef{eq:1.4}, we can show that the function $\sup_{t>0} | D^ku(x,t)|$ 
decay exponentially for $|k|\le 2$ (cf.\ the proof of Lemma \ref{lem:6.4}). 
Applying the Schauder estimate, we also have the uniform boundedness of 
$\| \nabla u(\cdot,t)\|_{L^{\infty}(\RN)}$.
Let $\{t_n\}_{n\in\N}$ be a sequence such that $t_n \to \infty$. 
Then by (i) of Lemma \ref{lem:2.1}, it follows that $\| u(\cdot,t_n)\|_{H^1} \le C$. 
Passing to a subsequence, we may assume that $u(\cdot,t_n) \rightharpoonup w$ in $H^1(\RN)$ for some $w\in H^1(\RN)$. 
Then arguing as in Lemma \ref{limite}, one can see that 
either $w=0$ or $w$ is a positive solution of \ef{eq:1.3}. 
In particular, $w$ decays exponentially at infinity. 
Arguing as for the proof of (i) of Lemma \ref{lem:2.3}, we have 
$\| u(\cdot,t_n)-w(\cdot) \|_{L^2} \to 0$. 
Let $U$ be any bounded domain. 
Then applying higher order regularity theory, we get 
$$
\| u(\cdot,t_n)-w(\cdot) \|_{H^m(V)}
\le C \| u(\cdot,t_n)-w(\cdot) \|_{L^2(U)} 
$$
for any $V \subset \subset U$ and $m \ge \frac{N}{2}+1$. 
By the Sobolev embedding $H^m(V) \hookrightarrow C^0(V)$, 
passing to a subsequence, we have that $u(\cdot,t_n) \to w$ uniformly on $V$. 
Since $U$ is arbitrary and $u(x,t_n)$ decays uniformly at infinity, 
it follows that $u(\cdot,t_n) \to w(\cdot)$ uniformly in $\RN$. 
Finally since $m \ge \frac{N}{2}+1$, we have the continuous embedding 
$H^m_{loc}(\RN) \hookrightarrow C^1_{loc}(\RN)$. 
Together with the uniform exponential decay of $|D^k(u(\cdot,t_n)|$ for $|k|\le 2$, 
passing to a subsequence if necessary, 
it follows that $u(\cdot,t_n) \to w$ in $C^1(\RN)$. This completes the proof.
\end{proof}

\noindent
To finish the proof of Theorem \ref{thm:1.1}, we have to prove \ef{eq:1.5} 
and show that the limit $w \in \Omega(u)$ is independent of the choice of the sequence $\{t_n\}$. 
To this end, we put 
$$
\eta(y,t):=\left( \int_0^T \| u(\cdot,s+t)-w(\cdot+y) \|_{H^1(\RN)}^2 \,ds \right)^{1\over 2},
$$
where $w$ is a fixed element in the $\omega$-limit set $\Omega(u)$. 
First we state the following proposition whose proof will be given later. 

\begin{proposition}\label{prop:4.2}
There exist $M>0$ and $T>1$ such that the following properties hold. 
For every sequence $\{(y_n,t_n)\}_{n\in\N} \subset \R^N \times \R^+$ satisfying 
$|y_n| \le 1$, $t_n \to \infty$ and 
$$
\eta(y_n,t_n) 
=\left( \int_0^T \| u(\cdot,s+t_n)-w(\cdot+y_n) \|_{H^1}^2 \,ds \right)^{1\over 2} \to 0,
$$
there exist a subsequence $\{(y_{n_j},t_{n_j})\}$ and 
$\{ z_j \}\subset \RN$ such that $| z_j| \le M\eta (y_{n_j},t_{n_j})$ and 
$$
\eta \left( z_{n_j}+y_{n_j},t_{n_j}+T \right) \le \frac{1}{2} \eta(y_{n_j},t_{n_j}).
$$
\end{proposition}

\noindent
By Proposition \ref{prop:4.2}, we obtain the following Corollary.

\begin{corollary}[Uniform stability]
	\label{cor:4.3}
There exist $M>0$, $T>1$, $t_0>0$ and $\eta_0>0$ such that the following property hold. 
For every $(y,t) \in \RN \times \R_+$ satisfying 
$|y| \le 1$, $t \ge t_0$ and $\eta(y,t) \le \eta_0$, 
there exists $z\in \RN$ such that $|z| \le M\eta(y,t)$ and 
$$
\eta(z+y,t+T) \le \frac{1}{2} \eta(y,t). 
$$
\end{corollary}

\begin{proof}
Let $M>0$, $T>1$ be the constants provided in Proposition \ref{prop:4.2}. 
We assume by contradiction that Corollary \ref{cor:4.3} does not hold. 
Then there exists $\{ (y_n,t_n)\}_{n\in\N} \subset \RN \times \R_+$ such that 
$\eta(y_n,t_n) \to 0$, $|y_n|\leq 1$, $t_n\to\infty$ and 
$$
\eta(z+y_n,t_n+T)>\frac{1}{2}\eta(y_n,t_n)
$$
for all $z\in \RN$ with $|z| \le M \eta (y_n,t_n)$. 
This contradicts Proposition \ref{prop:4.2}. 
\end{proof}

\begin{lemma} \label{lem:4.4}
Let $M>0$, $T>1$ and $t_0>0$ be constants provided by Corollary \ref{cor:4.3}. 
Then there exists $\bar{\eta}>0$ such that the following property hold. 
For every $k \in \N$ and $t^*>t_0$ with $\eta(0,t^*)\le \bar{\eta}$, 
there exists $\{ x_i \}_{i=1}^k \subset \RN$ such that 
$| x_i| \le M\eta \big( x_1+\cdots +x_{i-1},t^*+(i-1)T \big)$ and 
$$
\eta \big( x_1+\cdots+x_k,t^*+kT \big) 
\le \frac{1}{2} \eta \big( x_1+\cdots+x_{k-1},t^*+(k-1)T \big). 
$$
Here we put $x_0=0$. 
\end{lemma}

\begin{proof}
If $M>0$, $T>1$, $t_0>0$ and $\eta_0>0$ denote the constants by Corollary \ref{cor:4.3}, we define 
$$
\bar{\eta}:= \frac{1}{2} \min \left\{ \eta_0, \frac{1}{2M} \right\}
$$
and claim that for each $k\in \N$, there exists $x_k \in \RN$ such that 
\begin{equation}\label{eq:4.1}
\left\{
\begin{array}{l} \displaystyle 
\eta \big( x_1+\cdots+x_k,t^*+kT \big) 
\le \frac{1}{2} \eta \big( x_1+\cdots+x_{k-1},t^*+(k-1)T \big), \smallskip \\ 
|x_k| \le M \eta \big( x_1+\cdots +x_{k-1},t^*+(k-1)T \big), \smallskip \\ 
|x_1+\cdots +x_k|\le 1, \smallskip \\ 
\eta \big( x_1+\cdots +x_k,t^*+kT \big) <\bar{\eta}. 
\end{array}
\right.
\end{equation}
This will be proved by an induction argument. 
Suppose that $\eta(0,t^*) \le \bar{\eta}$. 
Then by applying Corollary \ref{cor:4.3} with $y:=0$ and $t:=t^*$, 
there exists $z=:x_1 \in \RN$ such that 
$$
|x_1| \le M\eta(0,t^*) \le M\bar{\eta} \le \frac{1}{4} 
\quad \hbox{and} \quad 
\eta(x_1,t^*+T) \le \frac{1}{2}\eta(0,t^*) < \bar{\eta}. 
$$
This implies that \ef{eq:4.1} holds for $k=1$. 
Next we assume that \ef{eq:4.1} holds for $k \in \N$. 
Using Corollary \ref{cor:4.3} with $y:=x_1+\cdots+x_k$ and $t:=t^*+kT$, there exists 
$x_{k+1}\in \RN$ such that $|x_{k+1}| \le M\eta (x_1+\cdots +x_k,t^*+kT)$ and 
\begin{equation}\label{eq:4.2}
\eta \big (x_1+\cdots+x_k+x_{k+1},t^*+(k+1)T \big) 
\le \frac{1}{2} \eta \big( x_1+\cdots +x_k,t^*+kT \big). 
\end{equation}
To finish the inductive step, it suffices to show that 
$$
|x_1+\cdots+x_k+x_{k+1}|\le 1 \quad \hbox{and} \quad 
\eta \big( x_1+\cdots+x_k+x_{k+1},t^*+(k+1)T \big) <\bar{\eta}. 
$$
Now by the induction hypothesis, it follows that 
$$
\sum_{i=2}^{k+1} |x_i|
\le M \sum_{i=2}^{k+1} \eta \big( x_1+\cdots+x_{i-1},t^*+(i-1)T \big), 
$$
\begin{equation}\label{eq:4.3}
\eta \big( x_1+\cdots+x_{i-1},t^*+(i-1)T \big) 
\le \frac{1}{2}\eta \big( x_1+\cdots+x_{i-2},t^*+(i-2)T \big) 
\le \frac{1}{2^{i-1}} \eta(0,t^*),
\end{equation}
for every $2\leq i\leq k+1$. Thus one has 
$$
|x_1+\cdots+x_k+x_{k+1}| \le 2M \eta(0,t^*) \le 2M\bar{\eta} \le \frac{1}{2}. 
$$
Finally from \ef{eq:4.1} and \ef{eq:4.2}, we also have 
$$
\eta \big( x_1+\cdots +x_k+x_{k+1},t^*+(k+1)T \big) \le \frac{1}{2}\bar{\eta}.
$$
Thus by induction, Lemma \ref{lem:4.4} holds. 
\end{proof}

\begin{lemma}\label{lem:4.5}
Let $T>1$, $t_0>0$ and $\eta_0>0$ be constants in Corollary \ref{cor:4.3} and Lemma \ref{lem:4.4}. 
Then there exists $\tilde{C}>0$ such that the following properties hold. 
For every $k\in \N$ and $t^*>t_0$ with $\eta(0,t^*) \le \bar{\eta}$, it follows that 
$$
\eta(0,t^*+kT) \le \tilde{C}\eta(0,t^*). 
$$
\end{lemma}

\begin{proof}
Let $M>0$, $T>1$, $t_0>0$, $\bar{\eta}>0$ and 
$\{ x_i\}_{i=1}^k \subset \RN$ be as in Lemma \ref{lem:4.4}. 
Then by the triangular inequality and (ii) of Lemma \ref{lem:2.5}, one has 
\begin{align*}
& \big| \eta(x_1+\cdots +x_k,t^*+kT)-\eta(0,t^*+kT) \big| \\ 
&\le \Biggl| \left( \int_0^T \| u(\cdot,s+t^*+kT)-w(\cdot+x_1+\cdots +x_k)
\|_{H^1}^2 \,ds \right)^{1\over 2} \\ 
& \qquad -\left( \int_0^T \| u(\cdot,s+t^*+kT)-w(\cdot) \|_{H^1}^2 \,ds 
\right)^{1\over 2} \Biggr| \\ 
&\le \left( \int_0^T \| w(\cdot +x_1+\cdots +x_k)-w(\cdot) \|_{H^1}^2
\,ds \right)^{1\over 2} \\ 
&\le CT^{1\over 2} \sum_{i=1}^{k} |x_i| 
\le CT^{1\over 2} M \sum_{i=1}^k \eta\big( x_1+\cdots +x_{i-1},t^*+(i-1)T\big). 
\end{align*}
Thus from \ef{eq:4.3}, we obtain 
\begin{align*}
\eta(0,t^*+kT)
&\le (1+CT^{1\over 2}M) \sum_{i=1}^k \eta\big( x_1+\cdots +x_{i-1},t^*+(i-1)T \big) \\ 
&\le 2(1+CT^{1\over 2}M)\eta(0,t^*). 
\end{align*}
Taking $\tilde{C}=2(1+CT^{1\over 2}M)$, the claim holds. 
\end{proof}

\noindent
We shall now prove Proposition \ref{prop:4.2}. 
Let $T>1$ be a constant which will be chosen later and suppose that a sequence $\{ (y_n,t_n) \}_{n\in\N} \subset \RN \times \R_+$ satisfies $|y_n| \le 1$, $t_n \to \infty$ and
$$
\eta(y_n,t_n)=
\left( \int_0^T \| u(\cdot,s+t_n)-w(\cdot+y_n) \|_{H^1(\RN)}^2 \,ds \right)^{1\over 2} \to 0. 
$$
Then passing to a subsequence, we may assume that $y_n \to y_0$ as $n \to \infty$. 
Moreover since $t_n \to \infty$ as $n\to\infty$, 
we may also assume that $u(\cdot,t_n) \to \tilde{w}$ uniformly for some $\tilde{w} \in \Omega(u)$. 
Thus for any $K>1$, we have by (ii) of Lemma \ref{lem:2.3} that 
$$
\lim_{ n \to \infty} \| u(\cdot,t+t_n)-\tilde{w}(\cdot) \|_{L^{\infty}(\RN \times [0,K])}=0. 
$$
On the other hand, it follows that $w(\cdot+y_n) \to w(\cdot+y_0)$ in $H^1(\RN)$. 
Thus from $\eta(y_n,t_n) \to 0$, 
\begin{equation}\label{eq:5.1}
\lim_{ n \to \infty} \int_0^T
\| u(\cdot,s+t_n)-w(\cdot+y_0) \|_{H^1(\RN)}^2 \,ds =0. 
\end{equation}
This implies that $\tilde{w}(\cdot)=w(\cdot +y_0)$ and hence 
\begin{equation}\label{eq:5.2}
\lim_{n \to \infty} \| u(\cdot,t+t_n)-w(\cdot+y_0) \|_{L^{\infty}(\RN \times [0,K])}=0. 
\end{equation}
Moreover by Lemma \ref{lem:3.1}, 
we know that $\{ u(\cdot, t+t_n)\}$ is relatively compact in $C^1(\RN)$. 
Thus by the uniform convergence of $u(\cdot,t+t_n) \to w(\cdot+y_0)$, one also has 
\begin{equation}
\label{gradienti}
\lim_{n \to \infty} \| \nabla u(\cdot,t+t_n)-\nabla w(\cdot+y_0) 
\|_{L^{\infty}(\RN \times [0,K])}=0. 
\end{equation}
Hereafter, we write for simplicity 
$$
\eta_n:=\eta(y_n,t_n),\quad 
u_n(x,t):=u(x,t+t_n),\quad 
w_n(x):=w(x+y_n),\quad
w_0(x):=w(x+y_0).
$$
We also note that up to translation, $w_0$ is radially symmetric with respect to $y_0$. 
Now we set
$$
\phi_n(x,t):=\frac{u_n(x,t)-w_n(x)}{\eta_n}. $$
Since 
$$
\int_{1\over 2}^1 \| \phi(\cdot,s) \|_{H^1}^2 \,ds 
\le \int_0^T \| \phi_n (\cdot,s) \|_{H^1}^2 \,ds =1, $$
we have $\| \phi_n(\cdot, \tau_n) \|_{H^1} \le 2$ 
for some $\{ \tau_n\} \subset [\frac{1}{2},1]$ by the mean value theorem. 
Thus, passing to a subsequence, we may assume that 
\begin{equation}\label{eq:5.3}
\tau_n \to \tau_0 \in \Big[\frac{1}{2},1\Big] \ \hbox{and} \ 
\phi_n(\cdot,\tau_n) \rightharpoonup \phi_0 \ \hbox{in} \ H^1(\RN),
\end{equation}
for some $\phi_0 \in H^1(\RN)$. 
Moreover by the compact embedding $H^1_{{\rm loc}}(\RN) \hookrightarrow L^2_{{\rm loc}}(\RN)$, 
we have
$$
\phi_n(\cdot,\tau_n) \to \phi_0 \quad \hbox{in} \ L_{{\rm loc}}^2(\RN) 
\quad \hbox{and} \quad \| \phi_0 \|_{L^2(\RN)} \le 2. 
$$

\begin{lemma}\label{lem:5.1}
Let $K>1$ be given. Then there exists $C>0$ such that
$$
\sup_{n\in\N}\int_0^K \| \phi_n(\cdot,s) \|_{H^1(\RN)}^2 \,ds \le C .
$$
\end{lemma}

\begin{proof}
By applying (ii) of Lemma \ref{lem:2.2} with $t_1=\tau_n+t_n$ and $t_2=K+t_n$, one has 
$$
\int_{\tau_n}^K \| \phi_n(\cdot,s) \|_{H^1}^2 \,ds 
\le Ce^{2C(K-\tau_n)} \| \phi_n(\cdot,\tau_n) \|_{L^2}^2. 
$$
Since $\tau_n \in [\frac{1}{2},1]$ and $\| \phi_n(\cdot,\tau_n) \|_{L^2(\RN)} \le 2$, we get 
$$
\int_{\tau_n}^K \| \phi_n(\cdot,s) \|_{H^1}^2 \,ds \le C,
$$
where $C>0$ is independent of $n \in \N$. 
Since $\int_0^T \| \phi_n(\cdot,s)\|_{H^1}^2 \,ds=1$, we also have 
$$
\int_0^{\tau_n} \| \phi_n(\cdot,s) \|_{H^1}^2 \,ds 
\le \int_0^T \| \phi_n(\cdot,s)\|_{H^1} \,ds =1. 
$$
Thus the claim holds. 
\end{proof}

\begin{lemma}[Convergence to the linearized problem]
	\label{lem:5.2}
Let $K>1$ be arbitrarily given. 
Then there exists a subsequence of $\{ \phi_n\}$, still denoted by $\{ \phi_n\}$, 
such that $\phi_n \rightharpoonup \phi$ in $L^2( [0,K),H^1(\RN))$. 
Moreover $\phi \in C( (0,\infty), L^2(\RN))$ and satisfies the following linear parabolic problem 
\begin{equation}
\label{eq:5.4}
\left\{
\begin{array}{ll}
\phi_t+{\mathcal L}_0 \phi =0 &\hbox{in} \ \RN \times (0,\infty), \\ 
\phi(x,\tau_0)=\phi_0(x) &\hbox{in} \ \RN. 
\end{array}
\right.
\end{equation}
Here ${\mathcal L}_0$ is the linearized operator around $w_0$, which is defined by 
$$
{\mathcal L}_0 \phi:= -(1+2w_0^2)\Delta \phi -4w_0 \nabla w_0 \cdot \nabla \phi 
-(4w_0\Delta w_0+2| \nabla w_0|^2) \phi +\phi -pw_0^{p-1} \phi. 
$$
\end{lemma}

\begin{proof}
The weak convergence of $\phi_n$ follows by Lemma \ref{lem:5.1}. 
We show that the weak limit $\phi$ satisfies \ef{eq:5.4}. 
Now from \ef{eq:1.1} and \ef{eq:1.3} and by the definition of $\phi_n$, one has 
\begin{equation}\label{eq:5.5}
(\phi_n)_t-\Delta \phi_n+\phi_n
-\eta_n^{-1}\big( u_n\Delta u_n^2-w_n\Delta w_n^2 \big)
-\eta_n^{-1}(u_n^p-w_n^p)=0. 
\end{equation}
Fix $\varphi \in C_0^{\infty}\big( \RN\times [0,K] \big)$. 
Multiplying \ef{eq:5.5} by $\varphi$ and integrating over $[\tau_n,K] \times \RN$, we get 
\begin{align*}
&\int_{\tau_n}^K \int_{\RN} \Big(
-\phi_n \varphi_t+\nabla \phi_n \cdot \nabla \varphi
+\phi_n \varphi
-\frac{1}{\eta_n}\big( u_n\Delta u_n^2-w_n\Delta w_n^2\big) \varphi 
-\frac{1}{\eta_n}(u_n^p-w_n^p) \varphi \Big) \,dx \,ds \\ 
&=\int_{\RN} \phi_n(x,\tau_n) \varphi(x,\tau_n) \,dx. 
\end{align*}
Now from the integration by parts, it follows that 
\begin{align}\label{eq:5.6}
-\frac{1}{\eta_n} \int_{\RN} \big( u_n\Delta u_n^2-w_n\Delta w_n^2 \big) \varphi \,dx 
&= 2 \int_{\RN} ( u_n \nabla (u_n+w_n)\cdot \nabla \phi_n 
+| \nabla w_n|^2 \phi_n)\varphi \,dx \nonumber \\
&+2\int_{\RN} \big( (u_n+w_n)\phi_n \nabla u_n +w_n^2 \nabla \phi_n\big) \cdot \nabla \varphi \,dx.
\end{align}
Moreover by the mean value theorem, we also have
\begin{equation}\label{eq:5.7}
-\frac{1}{\eta_n}\int_{\RN} (u_n^p-w_n^p) \varphi \,dx 
=-p \int_{\RN} (\kappa_n u_n+(1-\kappa_n)w_n)^{p-1} \phi_n \varphi \,dx 
\end{equation}
for some $\kappa_n \in (0,1)$. 
Thus, we obtain 
\begin{equation}\label{eq:5.8}
 \int_{\tau_0}^K \int_{\RN} (-\phi_n\varphi_t 
+{\mathcal L}_0 \phi_n \cdot \varphi ) \,dx \,ds 
-\int_{\RN} \phi_0(x) \varphi(x,\tau_0) \,dx=\sum_{i=1}^4 \I_i^n,
\end{equation}
where we have set 
\begin{align*}
 \I^n_1& := \int_{\tau_0}^{\tau_n} \int_{\RN}
\Big\{ -\phi_n \varphi_t+\nabla \phi_n \cdot \nabla \varphi
+\phi_n \varphi
+2 \big( (u_n+w_n)\phi_n \nabla u_n +w_n^2 \nabla \phi_n\big) \cdot \nabla \varphi
\nonumber \\
&\quad +2( u_n \nabla (u_n+w_n)\cdot \nabla \phi_n
+| \nabla w_n|^2 \phi_n)\varphi 
-p(\kappa_n u_n+(1-\kappa_n)w_n)^{p-1}\phi_n \varphi 
\Big\} \,dx \,ds, \\
 \I^n_2&:=\int_{\RN} \big( \phi_n(x,\tau_n)\varphi(x,\tau_n)-\phi_0(x) \varphi(x,\tau_0) \big) \,dx, \\
\I^n_3 & :=\int_{\tau_0}^K \int_{\RN}
p \big((\kappa_n u_n+(1-\kappa_n)w_n)^{p-1}- w_0^{p-1}\Big) 
\phi_n \varphi \,dx \,ds, \\
\I^n_4& :=-\int_{\tau_0}^K \int_{\RN} \Big\{
2 \phi_n \big( (u_n+w_n) \nabla u_n- 2 w_0 \nabla w_0 \big) \cdot \nabla \varphi +2(w_n^2-w_0^2) \nabla \phi_n \cdot \nabla \varphi 
 \nonumber \\
&\quad
+2\varphi \big( u_n \nabla (u_n+w_n) -2w_0 \nabla w_0 \big) \cdot \nabla \phi_n
+2(|\nabla w_n|^2 -|\nabla w_0|^2) \phi_n \varphi  \Big\} \,dx \,ds.
\end{align*}
Now by virtue of Lemma \ref{lem:5.1}, as $n\to\infty$ one has 
\begin{align*}
|\I_1^n| &\le 
C \int_{\tau_0}^{\tau_n} \int_{\RN}
\Big( |\phi_n | \big( |\varphi_t|+|\varphi|+|\nabla \varphi| \big)
+|\nabla \phi_n| \big( |\nabla \varphi|+|\varphi| \big) \Big) \,dx \,ds \\
&\le C \Big( \int_{\tau_0}^{\tau_n} \int_{\RN} 
\big(|\varphi_t|^2+|\varphi|^2+|\nabla \varphi|^2 \big) dx \,ds \Big)^{1\over 2}
\Big( \int_{\tau_0}^{\tau_n} \int_{\RN} \big(|\nabla \phi_n|^2 +|\phi_n|^2\big) \,dx \,ds\Big)^{1\over 2} \\
&\le C |\tau_n -\tau_0|^{1\over 2} 
\Big( \int_0^K \| \phi_n(\cdot,s) \|_{H^1}^2 \,ds \Big)^{1\over 2}
\le C| \tau_n -\tau_0|^{1\over 2} \to 0.
\end{align*}
Moreover since $\phi_n(\cdot,\tau_n) \to \phi_0$ in $L^2_{{\rm loc}}(\RN)$, 
we also have $|\I_2^n| \to 0$ as $n \to \infty$. 
Next from \ef{eq:5.2} and by the uniform convergence of $w_n$ to $ w_0$, it follows that 
$$
\lim_{n \to \infty} \Big( \sup_{\RN \times [0,K]}
\big| w_0^{p-1}-
(\kappa_n u_n +(1-\kappa_n)w_n)^{p-1} \big| \Big) =0. 
$$
Thus by Lemma \ref{lem:5.1}, one has $|\I_3^n| \to 0$. 
Similarly by the uniform convergences of $\nabla u_n \to \nabla w_0$, 
$\nabla w_n \to \nabla w_0$ and from \ef{eq:5.1}, we also have $|\I_4^n| \to 0$. 
Letting $n \to \infty$ in \ef{eq:5.8}, we obtain 
$$
\int_{\tau_0}^K \int_{\RN} 
(-\phi \varphi_t+{\mathcal L}_0 \phi \cdot \varphi) \,dx \,ds 
-\int_{\RN} \phi_0(x) \varphi(x,\tau_0) \,dx =0. 
$$
This implies that $\phi$ is a weak solution of \ef{eq:5.4}. 
Then by the linear parabolic theory, it follows that $\phi$ is a classical solution and $\phi \in C \big( (0,\infty),L^2(\RN) \big)$. 
\end{proof}

\begin{lemma}
	\label{lem:5.3} 
Let $\theta_n(x,t):= \phi_n(x,t)-\phi(x,t).$ Then the following facts hold.
\begin{itemize}
\item[\rm(i)] Let $K>1$. 
Then there exist $n_1=n_1(K)\in \N$ and a positive constant $\hat{C}$ independent of $n \in \N$ and $K$ such that 
$$
\sup_{n\geq n_1}\int_1^K \| \theta_n(\cdot,s) \|_{H^1(\RN)}^2 \,ds\le \hat{C}. 
$$
\item[\rm(ii)] 
For any $\varepsilon>0$, there exist $T_{\varepsilon}>1$ and 
$n_2=n_2(T_{\varepsilon})\in \N$ such that
$$
\sup_{n\geq n_2}\int_{T_{\varepsilon}}^{2T_{\varepsilon}}
\| \theta_n(\cdot,s) \|_{H^1(\RN)}^2 \,ds < \varepsilon. 
$$
\end{itemize}
\end{lemma}

\begin{proof}
(i) Let $R>0$ be given. First, we claim that 
\begin{equation}\label{eq:5.9}
\lim_{n \to \infty} \sup_{t\in [\tau_n,K]} 
\| \theta_n(\cdot,t)\|_{L^2 \big( B(0,R) \big) }=0. 
\end{equation}
Now from \ef{eq:5.4} and \ef{eq:5.5}, one has 
\begin{align}\label{eq:5.10}
&(\theta_n)_t-\Delta \theta_n +\theta_n 
-\frac{1}{\eta_n}\big( u_n\Delta u_n^2-w_n\Delta w_n^2 \big) 
-\frac{1}{\eta_n}(u_n^p-w_n^p) \nonumber \\ 
&\quad +2w_0^2 \Delta \phi +4 w_0 \nabla w_0 \cdot \nabla \phi +4w_0 \Delta w_0 \phi 
+2|\nabla w_0|^2 \phi +pw_0^{p-1} \phi =0. 
\end{align}
Let $\xi \in C_0^{\infty}(\RN)$ be a cut-off function satisfying $\xi \equiv 1$ on $B(0,R)$. 
We multiply \ef{eq:5.10} by $\xi^2 \theta_n$ and integrate it over $\RN$. 
Then, from \ef{eq:5.6}, \ef{eq:5.7} and the integration by parts, we get 
\begin{equation}\label{eq:5.11}
\int_{\RN} \Big(\frac{1}{2} \frac{\partial}{\partial t} (\xi^2 \theta_n^2)
+(1+2w_0^2)|\nabla \theta_n|^2 \xi^2 +\theta_n^2 \xi^2
+2 \xi \theta_n \nabla \theta_n \cdot \nabla \xi \Big)\,dx
=-\sum_{i=1}^5  \I_i^n,
\end{equation}
where we have set 
\begin{align*}
 \I^n_1&:= \int_{\RN} \Big( 2u_n \nabla (u_n+w_n) \cdot \nabla \phi_n
-4w_0 \nabla w_0 \cdot \nabla \phi \Big) \xi^2 \theta_n \,dx, \\
\I^n_2& :=\int_{\RN}\big( 2|\nabla w_n|^2 \phi_n -2|\nabla w_0|^2 \phi \Big) 
\xi^2 \theta_n  \,dx,  \\
\I^n_3& := \int_{\RN} \Big( 2(u_n+w_n)\phi_n \nabla u_n
-4w_0\phi \nabla w_0 \Big) \cdot \nabla (\xi^2 \theta_n) \,dx,  \\
\I^n_4& :=\int_{\RN} \Big( 2(w_n^2-w_0^2) \xi^2 \nabla \phi_n \cdot \nabla \theta_n +2(w_n^2\nabla \phi_n-w_0^2 \nabla \phi) 
\cdot \nabla (\xi^2)\theta_n \Big)\,dx, \\
\I^n_5& := \int_{\RN} p
\Big( w_0^{p-1} \phi -(\kappa_n u_n +(1-\kappa_n)w_n)^{p-1} \phi_n 
\Big) \xi^2 \, \theta_n \,dx.
\end{align*}
Now by the Schwarz and the Young inequalities, it follows that 
$$
2 \int_{\RN} \xi \theta_n \nabla \theta_n \cdot \nabla \xi \,dx
\ge -\frac{1}{2} \int_{\RN} |\nabla \theta_n|^2 |\xi|^2 \,dx
-2 \int_{\RN} |\theta_n|^2 |\nabla \xi|^2 \,dx.$$
Next by the Schwarz inequality, one has 
\begin{align*}
|\I_1^n| &\le 2 \| u_n \nabla (u_n+w_n) \|_{L^{\infty}(\RN \times [0,K])} \| \xi \nabla \theta_n \|_{L^2} \| \xi \theta_n \|_{L^2} \\ 
& \quad +2 \| u_n \nabla (u_n+w_n) -2w_0 \nabla w_0 \|_{L^{\infty}(\RN \times [0,K])} \| \xi \nabla \phi \|_{L^2} \| \xi \theta_n \|_{L^2}. 
\end{align*}
Similarly we have 
\begin{align*}
|\I_2^n| &\le 2 \left\| |\nabla w_n|^2-|\nabla w_0|^2 \right\|_{L^{\infty}}
\| \xi \theta_n \|_{L^2} \| \phi_n \|_{L^2}
+2 \| \nabla w_0 \|_{L^{\infty}}^2 \| \xi \theta_n \|_{L^2}^2, \\
|\I_3^n| &\le 2 \| (u_n+w_n)\nabla u_n -2w_0 \nabla w_0 \|_{L^{\infty}}
\| \phi_n \|_{L^2} \| \xi \nabla \theta_n \|_{L^2} \\
&\quad +4\| (u_n+w_n)\nabla u_n-2w_0 \nabla w_0 \|_{L^{\infty}} 
\| \xi \theta_n \|_{L^2} \| \phi_n \nabla \xi\|_{L^2}, \\
&\quad +4 \| w_0 \nabla w_0\|_{L^{\infty}} \| \xi \theta_n \|_{L^2} \| \xi \nabla \theta_n \|_{L^2}
+ 8 \| w_0 \nabla w_0\|_{L^{\infty}} \| \xi \theta_n \|_{L^2} \| \theta_n \nabla \xi \|_{L^2} \\
|\I_4^n| &\le 2 \| w_n^2 -w_0^2 \|_{L^{\infty}} \| \xi \nabla \phi \|_{L^2}
\| \xi \nabla \theta_n \|_{L^2} +2 \| w_n^2 -w_0^2 \|_{L^{\infty}}
\| \xi \nabla \theta_n \|_{L^2}^2 \\
&\quad +4 \| w_n^2-w_0^2 \|_{L^{\infty}} \| \xi \nabla \phi \|_{L^2} 
\| \theta_n \nabla \xi \|_{L^2}
+4 \| w_n \|_{L^{\infty}}^2 \| \xi \nabla \theta_n \|_{L^2}
\| \theta_n \nabla \xi \|_{L^2},\\
|\I_5^n| &\le p \| w_0^{p-1}-(\kappa_n u_n +(1-\kappa_n)w_n)^{p-1}\|_{L^{\infty}}
\| \phi_n \|_{L^2} \| \xi \theta_n \|_{L^2}
+p\| w_0 \|_{L^{\infty}}^{p-1} \| \xi \theta_n \|_{L^2}^2.
\end{align*}
Next applying (i) of Lemma \ref{lem:2.2} with $t_1=\tau_n+t_n$ and $t_2=t+t_n$, we get 
$$
\| \phi_n(\cdot,t) \|_{L^2}
\le e^{C(t-\tau_n)} \| \phi_n(\cdot,\tau_n) \|_{L^2}
\le Ce^{CK} \quad \hbox{for} \ t\in [\tau_n,K] 
$$
and hence $\| \theta_n(\cdot,t) \|_{L^2} \le C$. 
Thus from \ef{eq:5.11}, the uniform decays of $u_n$, $w_n$, $\nabla u_n$, $\nabla w_n$ and by the Young inequality, we obtain 
\begin{align*}
&\frac{\partial }{\partial t} \| \xi \theta_n \|_{L^2}^2
+\int_{\RN} (|\nabla \theta_n|^2+|\theta_n|^2) |\xi|^2 \,dx \\
&\le C \| \xi \theta_n \|_{L^2}^2
+C \| \theta_n \nabla \xi \|_{L^2}^2 \\
&\quad +2\| w_n^2-w_0^2 \|_{L^{\infty}} \| \xi \nabla \theta_n\|_{L^2}^2
+C \| \theta_n \nabla \xi \|_{L^2} +h_n, 
\end{align*}
where $C$ and $h_n$ are positive constants with $h_n \to 0$. 
Now let $\varepsilon>0$. 
We choose $\xi$ so that 
$$
C \| \theta_n \nabla \xi \|_{L^2}^2 +C \| \theta_n \nabla \xi \|_{L^2} \le C \sup_{\RN} |\nabla \xi|(1+|\nabla \xi|) < \frac{\varepsilon}{2}. 
$$
Next we take large $n_0 \in \N$ so that
$$
2\| w_n^2-w_0^2 \|_{L^{\infty}} \le \frac{1}{2} \ \ 
\hbox{and} \ \ h_n < \frac{\varepsilon}{2} \ \hbox{for} \ n \ge n_0. 
$$
Then we obtain 
\begin{equation}\label{eq:5.12}
\frac{\partial }{\partial t} \| \xi \theta_n(\cdot,t) \|_{L^2}^2
\le C \| \xi \theta_n(\cdot,t) \|_{L^2}^2 +\varepsilon. 
\end{equation}
Let $\zeta_n(t):= \| \xi \theta_n(\cdot,t)\|_{L^2}^2$. 
From \ef{eq:5.12}, it follows that $\zeta_n' \le C \zeta_n+\varepsilon$. 
Thus by the Gronwall inequality, one has 
$$
\zeta_n(t) \le e^{C(t-\tau_n)}\zeta_n(\tau_n)+e^{Ct}\int_{\tau_n}^t \varepsilon e^{-Cs} \,ds 
\le e^{CK}\left( \zeta_n(\tau_n)+ \frac{\varepsilon}{C}\right) \ \hbox{for} \ t\in [\tau_n,K]. 
$$
Since $\phi_n(\cdot,\tau_n) \to \phi(\cdot,\tau_0)$ in $L^2_{{\rm loc}}(\RN)$, 
we have $\zeta_n(\tau_n)=\| \xi \theta_n(\cdot,\tau_n) \|_{L^2}^2 \to 0$. Thus, 
$$ 
\limsup_{n \to \infty} \Big( \sup_{t \in [\tau_n,K]}
\| \theta_n(\cdot,t) \|_{L^2 \big( B(0,R \big)} \Big) \le \frac{\varepsilon}{C}e^{CK}. 
$$
Since $\varepsilon$ is arbitrarily, \ef{eq:5.9} holds. 
Next we show that $\int_1^K \| \theta_n(\cdot,s) \|_{H^1(\RN)}^2 \,ds \le \hat{C}$. 
To this aim, we multiply \ef{eq:5.10} by $\theta_n$ and integrate it over $\RN$. 
Then arguing as above, one has 
\begin{equation*}
\int_{\RN} \Big(\frac{1}{2} \frac{\partial}{\partial t} (\theta_n)^2
+(1+2w_0^2)|\nabla \theta_n|^2 +|\theta_n|^2\Big) \,dx= - \sum_{i=1}^4 \J^n_i, 
\end{equation*}
where we have set 
\begin{align*}
\J^n_1 &:= \int_{\RN} \Big( 2u_n \nabla (u_n+w_n)\cdot \nabla \phi_n
-4 w_0 \nabla w_0 \cdot \nabla \phi \Big) \theta_n \,dx, \\
\J^n_2& :=\int_{\RN} \Big( 2(u_n+w_n)\phi_n \nabla u_n-4w_0\phi \nabla w_0 \Big)
\cdot \nabla \theta_n \,dx, \\
\J^n_3& :=\int_{\RN} (2|\nabla w_n|^2 \phi_n-2|\nabla w_0|^2 \phi )\theta_n 
+2(w_n^2-w_0^2)\nabla \phi_n \cdot \nabla \theta_n \,dx, \\
\J^n_4&:=\int_{\RN} p \Big( w_0^{p-1}\phi -(\kappa_n u_n +(1-\kappa_n)w_n )^{p-1}
\phi_n \Big) \theta_n \,dx. 
\end{align*}
Now, we fix $\delta>0$ arbitrarily. 
By the Young inequality, it follows that 
\begin{align*}
|\J_1^n| &\le 2 \int_{\RN} |u_n\nabla(u_n+w_n)| |\theta_n| |\nabla \theta_n|
+|u_n\nabla (u_n+w_n)-2w_0 \nabla w_0| |\nabla \phi| |\theta_n| \,dx \\
&\le \frac{1}{8} \| \nabla \theta_n \|_{L^2}^2
+C \int_{\RN} |u_n\nabla (u_n+w_n)|^2 |\theta_n|^2 \,dx \\
&\quad +\delta \| \theta_n \|_{L^2}^2 
+C_{\delta} 
\| u_n\nabla (u_n+w_n)-2w_0 \nabla w_0 \|_{L^{\infty}}^2
\| \nabla \phi\|_{L^2}^2 \\
&\le \frac{1}{8} \| \nabla \theta_n \|_{L^2}^2 
+C \sup_{|x| \ge R} |u_n(x,t)| \int_{B^c(0,R)} |\theta_n|^2 \,dx
+C \int_{B(0,R)} |\theta_n|^2 \,dx \\
&\quad +\delta \| \theta_n \|_{L^2}^2 +C_{\delta} 
\| u_n\nabla (u_n+w_n)-2w_0 \nabla w_0 \|_{L^{\infty}}^2
\| \nabla \phi\|_{L^2}^2.
\end{align*}
From \ef{eq:1.4}, there exists $R_{\delta}>0$ such that 
$$
C \sup_{|x| \ge R_{\delta}} |u_n(x,t)| <\delta,
$$
for all $n \in \N$ and $t\in [1,K]$. 
Thus we obtain 
$$
|\J_1^n| \le \frac{1}{8} \| \nabla \theta_n \|_{L^2}^2 
+2\delta \| \theta_n \|_{L^2}^2
+C_{\delta} \int_{B(0,R_{\delta})} | \theta_n|^2 \,dx 
+C_{\delta} \hat{h}_n,
$$
where $C_{\delta}$ is a positive constant independent of $n\in \N$ and $K$, 
and $\hat{h}_n$ is a positive constant satisfying $\hat{h}_n \to 0$ as $n \to \infty$. 
Estimating $\J_2^n,\J_3^n,\J_4^n$ similarly, we have 
\begin{align*}
\frac{\partial }{\partial t} \| \theta_n\|_{L^2}^2 + \| \theta_n \|_{H^1}^2 
&\le \left( \frac{1}{2} + \| w_n^2-w_0^2 \|_{L^{\infty}} \right)
\| \nabla \theta_n \|_{L^2}^2 \\
&\quad +5 \delta \| \theta_n \|_{L^2}^2 
+C_{\delta} \int_{B(0,R_{\delta})} | \theta_n|^2 \,dx +C_{\delta} \hat{h}_n.
\end{align*}
Now we choose $\delta=1/10$. 
Taking $n\in \N$ larger if necessary, 
we have $\| w_n^2-w_0^2 \|_{L^{\infty}} \le 1/4$. 
Then we obtain 
\begin{equation}\label{eq:5.13}
\frac{\partial}{\partial t} \| \theta_n(\cdot,t) \|_{L^2}^2 
+\| \theta_n(\cdot,t) \|_{H^1}^2 
\le C_{\delta} \int_{B(0,R_{\delta})} |\theta_n(x,t)|^2 \,dx+C_{\delta} \hat{h}_n. 
\end{equation}
Integrating \ef{eq:5.13} over $[\tau_n,K]$, we get 
\begin{align*}
\int_{\tau_n}^K \| \theta_n(\cdot,s) \|_{H^1}^2 \,ds 
\le \| \theta_n(\cdot,\tau_n)\|_{L^2}^2 
+C_{\delta} \int_{\tau_n}^K \int_{B(0,R_{\delta})} | \theta_n(x,s)|^2 \,dx \,ds 
+C_{\delta} \hat{h}_n(K-\tau_n). 
\end{align*}
From \ef{eq:5.9} and $\hat{h}_n \to 0$, there exists $n_1=n_1(K) \in \N$ such that
$$
C_{\delta} \int_{\tau_n}^K \int_{B(0,R_{\delta})} | \theta_n(x,s)|^2 \,dx \,ds \le 2
\quad \hbox{and} \quad C_{\delta} \hat{h}_n(K-\tau_n) \le 2 \quad 
\hbox{for} \ n \ge n_1. 
$$
Moreover from $\| \phi_n(\cdot,\tau_n)\|_{L^2} \le 2$, 
$\| \phi(\cdot,\tau_0) \|_{L^2}\le 2$ and by the continuity of $\phi$, we also have 
\begin{align*}
\sup_{n \ge n_1} \| \theta_n(\cdot,\tau_n)\|_{L^2}^2 
\le \big( \| \phi_n(\cdot,\tau_n)\|_{L^2}
+\| \phi(\cdot,\tau_n)-\phi(\cdot,\tau_0)\|_{L^2}
+\| \phi( \cdot,\tau_0)\|_{L^2} \big)^2 \le 36. 
\end{align*}
Since $\tau_n \le 1$, we obtain
$$
\sup_{n \ge n_1} \int_1^K \| \theta_n(\cdot,s) \|_{H^1}^2 \,ds \le 40 
\quad \hbox{for} \ n \ge n_1. 
$$
This completes the proof of (i). 

(ii) We fix $\varepsilon>0$ arbitrarily and let $T>1$. 
First we observe from (i) that 
$$
\int_1^T \| \theta_n(\cdot,s)\|_{H^1}^2 \,ds \le \hat{C}
\quad \hbox{for} \ n \ge n_1(T). 
$$
Thus by the mean value theorem, there exists $s_n \in [1,T]$ such that 
$\| \theta_n(\cdot,s_n)\|_{L^2}^2 \le \frac{\hat{C}}{T-1}$. 
Next we integrate \ef{eq:5.13} over $[s_n,2T]$. 
Then from $\tau_n \le 1 \le s_n \le T$, it follows that 
\begin{align*}
\int_T^{2T} \| \theta_n(\cdot,s)\|_{H^1}^2 \,ds 
&\le \int_{s_n}^{2T} \| \theta_n(\cdot,s)\|_{H^1}^2 \,ds \\
&\le \| \theta_n(\cdot,s_n)\|_{L^2}^2
+C \int_{s_n}^{2T} \int_{B(0,R_{\delta})} |\theta_n|^2 \,dx\,ds \\
&\le \frac{\hat{C}}{T-1} 
+C \int_{\tau_n}^{2T} \int_{B(0,R_{\delta})} |\theta_n|^2 \,dx \,ds
+C \hat{h}_n (2T-s_n). 
\end{align*}
Now we choose $T_{\varepsilon}>1$ so that 
$\frac{\hat{C}}{T_{\varepsilon}-1} < \frac{\varepsilon}{3}$. 
Next from formula \ef{eq:5.9} and $\hat{h}_n \to 0$, 
we can take large $n_2=n_2(T_{\varepsilon})\in \N$ so that
$$
C \int_{\tau_n}^{2T_{\varepsilon}} \int_{B(0,R_{\delta})} 
|\theta_n|^2 \,dx \,ds < \frac{\varepsilon}{3} \quad 
\hbox{and} \quad C\hat{h}_n(2T_{\varepsilon}-s_n) < \frac{\varepsilon}{3}
\quad \hbox{for} \ n \ge n_2. 
$$
Then it follows that
$$
\sup_{n \ge n_2} \int_{T_{\varepsilon}}^{2T_{\varepsilon}} \| \theta_n(\cdot,s) \|_{H^1}^2 \,ds 
< \varepsilon 
$$
and hence the proof is complete. 
\end{proof}

\noindent
Now we consider the following eigenvalue problem
$$
{\mathcal L}_0 \psi =\mu \psi, \quad 
\psi \in L^2(\RN) \quad \hbox{and} \quad 
\psi(x) \to 0 \ \hbox{as} \ |x| \to \infty. 
$$
Then the first eigenvalue $\mu_1$ is negative. 
We denote by $\psi_1$ the associated eigenfunction with $\| \psi_1 \|_{L^2(\RN)}=1$. 
Moreover we know that the second eigenvalue $\mu_2$ is zero and the corresponding eigenspace is spanned by $\{ \frac{\partial w_0}{\partial x_i}\}_{i=1}^N$ (See \cite{ASW2}, Remark 4.10). 
Let $\tau_0 \in [\frac{1}{2},1]$ be as in \ef{eq:5.3} and decompose 
\begin{equation}\label{eq:5.14}
\phi_0(x)=\phi(x,\tau_0)= 
C_0 e^{-\mu_1 \tau_0} \psi_1(x) +\sum_{i=1}^N C_i \frac{\partial w_0}{\partial x_i}(x)
+\tilde{\theta}(x,\tau_0), 
\end{equation}
where $C_0, C_i \in \R$ and $\psi_1$, $\frac{\partial w_0}{\partial x_i}$, $\tilde{\theta}$ are mutually orthogonal in $L^2(\RN)$. 
Finally we set
\begin{equation}\label{eq:5.15}
\tilde{\theta}(x,t):=
\phi(x,t)- C_0 e^{-\mu_1 t} \psi_1(x)-
\sum_{i=1}^N C_i \frac{\partial w_0}{\partial x_i}(x). 
\end{equation}
Then by direct calculations, one can see that $\tilde{\theta}$ satisfies 
\begin{equation}\label{eq:5.16}
\tilde{\theta}_t+{\mathcal L}_0 \tilde{\theta}=0 \quad \hbox{in} \ 
\RN \times (0,\infty). 
\end{equation}
Moreover by the definition of $\tilde{\theta}$, we also have 
\begin{equation}\label{eq:5.17}
\int_{\RN} \tilde{\theta}(\cdot,\tau_0) \psi_1 \,dx
= \int_{\RN} \tilde{\theta}(\cdot,\tau_0) 
\frac{\partial w_0}{\partial x_i} \,dx
=0 \quad \hbox{for} \ i=1,\cdots,N. 
\end{equation}

\noindent
In the next result, we shall use the crucial information of non-degeneracy of stationary solutions. 

\begin{lemma}[Non-degeneracy and stability]
	\label{lem:5.4}
There exist $\alpha>0$ and $\tilde{T}>0$ such that
$e^{-\alpha \tilde{T}}<1$ and 
$$
\int_T^{2T} \| \tilde{\theta}(\cdot,s)\|_{H^1(\RN)}^2\,ds \le e^{-\alpha T}
\quad \hbox{for all} \ \ T \ge \tilde{T}. 
$$
\end{lemma}

\begin{proof}
First we claim that 
\begin{equation}\label{eq:5.18}
\int_{\RN} \tilde{\theta}(\cdot,t)\psi_1 \,dx
=\int_{\RN} \tilde{\theta}(\cdot,t) \frac{\partial w_0}{\partial x_i}\,dx
=0 \quad \hbox{for} \ i=1,\cdots,N \ \hbox{and} \ t \ge \tau_0. 
\end{equation}
To this aim, we put $\eta(t):= \int_{\RN} \tilde{\theta}(\cdot,t) \psi_1 \,dx$. 
Then from \ef{eq:5.16}, one has 
\begin{equation*}
\eta'(t)= ( \tilde{\theta}_t, \psi_1)_{L^2}
= -( {\mathcal L}_0 \tilde{\theta},\psi_1)_{L^2}
= -(\tilde{\theta}, {\mathcal L}_0 \psi_1 )_{L^2} 
= -\mu_1 (\tilde{\theta},\psi_1)_{L^2}
=-\mu_1 \eta. 
\end{equation*}
Thus from \ef{eq:5.17}, it follows that $\eta(t)=\eta(\tau_0)
e^{-\mu_1(t-\tau_0)}=0$ for all $t \ge \tau_0$.
We can prove the second equality in a similar way.
Next we define
\begin{align*}
\bar{\mu} &:=
\inf \Big\{
({\mathcal L}_0\psi,\psi)_{L^2} : \ 
\psi\in \HN(\RN), \ \| \psi\|_{\LN}=1, \\
&\hspace{5em} 
(\psi,\psi_1)_{L^2}
=\big(\psi, \frac{\partial w_0}{\partial x_i}\big)_{L^2}=0
\quad \hbox{for} \ i=1,\cdots, N \Big\}
\end{align*}
and suppose that $\bar{\mu}>0$ for the present. 
Then from \ef{eq:5.16} and \ef{eq:5.18}, we get 
\begin{equation}\label{eq:5.19}
\frac{d}{dt} \int_{\RN} \tilde{\theta}^2(\cdot,t) \,dx
=-2({\mathcal L}_0\tilde{\theta},\tilde{\theta})_{L^2}
\le - 2\bar{\mu} \int_{\RN} \tilde{\theta}^2(\cdot,t) \,dx
\end{equation}
and hence 
\begin{equation}\label{eq:5.20}
\int_{\RN} \tilde{\theta}^2(\cdot,t) \,dx
\le e^{-2\bar{\mu}(t-\tau_0)}
\int_{\RN} \tilde{\theta}^2(\cdot,\tau_0) \,dx
\quad \hbox{for} \ t \ge \tau_0. 
\end{equation}
Now let $\tilde{T}>1$ be a constant which will be chosen later 
and take $T \ge \tilde{T}$ arbitrarily. 
Integrating \ef{eq:5.19} over $[T,2T]$, one has 
$$
\int_{\RN} \tilde{\theta}^2(\cdot,2T) \,dx
+2\int_T^{2T} ({\mathcal L}_0\tilde{\theta},\tilde{\theta})_{L^2} \,ds
=\int_{\RN} \tilde{\theta}^2(\cdot,T) \,dx.
$$
Moreover by the Young inequality, we also have 
\begin{align}\label{eq:5.21}
({\mathcal L}_0 \tilde{\theta},\tilde{\theta})_{L^2} 
&= \int_{\RN} \Bigl( (1+2w_0^2)|\nabla \tilde{\theta}|^2
+8 w_0 \tilde{\theta} \nabla w_0 \cdot \nabla \tilde{\theta}
+2 \tilde{\theta}^2 |\nabla w_0|^2 
+\tilde{\theta}^2 -pw_0^{p-1} \tilde{\theta}^2 \Bigr) \,dx \nonumber \\
&\ge \int_{\RN} \Big(
|\nabla \tilde{\theta}|^2+\tilde{\theta}^2
-8w_0 |\tilde{\theta}| |\nabla w_0| |\nabla \tilde{\theta}|
-pw_0^{p-1}\tilde{\theta}^2 \Big) \,dx \nonumber \\
&\ge \frac{1}{2} \int_{\RN} | \nabla \tilde{\theta}|^2+\tilde{\theta}^2 \,dx
-\int_{\RN} \big( pw_0^{p-1}+32 w_0^2 |\nabla w_0|^2 \big) \tilde{\theta}^2 \,dx\\
&\ge \frac{1}{2} \| \tilde{\theta}(\cdot,t)\|_{\HN}^2
-C \| \tilde{\theta}(\cdot,t)\|_{\LN}^2. \nonumber 
\end{align}
Thus from \ef{eq:5.20}, we obtain 
\begin{align*}
\int_T^{2T} \| \tilde{\theta}(\cdot,s) \|_{\HN}^2 \,ds
&\le C \int_T^{2T} \| \tilde{\theta}(\cdot,s) \|_{\LN}^2 \,ds
+\| \tilde{\theta}(\cdot,T) \|_{\LN}^2 \\
&\le Ce^{2\bar{\mu} \tau_0} \| \tilde{\theta}(\cdot,\tau_0)\|_{L^2}^2
\int_T^{2T} e^{-2\bar{\mu}s} \,ds
+e^{2\bar{\mu}\tau_0} \| \tilde{\theta}(\cdot,\tau_0)\|_{L^2}^2
e^{-2\bar{\mu}T} \\
&\le \frac{C}{\bar{\mu}} e^{2\bar{\mu}\tau_0} \| \tilde{\theta}(\cdot,\tau_0)\|_{L^2}^2
e^{-2\bar{\mu}T}
+e^{2\bar{\mu}\tau_0} \| \tilde{\theta}(\cdot,\tau_0)\|_{L^2}^2
e^{-2\bar{\mu}T}\\
&\le \bar{C} e^{-2\bar{\mu}T} \quad 
\hbox{for all} \ \ T \ge \tilde{T},
\end{align*}
where $\bar{C}>0$ is independent of $T$ and $\tilde{T}$. 
Putting $\alpha :=\bar{\mu}>0$ and taking $\tilde{T}>1$ larger so that 
$\bar{C}e^{-\alpha \tilde{T}} \le 1$, the claim holds. 
We now show that $\bar{\mu}>0$. 
By the definition of $\bar{\mu}$ and $\mu_2=0$, it follows that $\bar{\mu} \ge 0$. 
Suppose by contradiction that $\bar{\mu}=0$. 
Then there exists $\{ \psi_n \} \subset \HN(\RN)$ such that $\| \psi_n\|_{L^2}=1$, 
$(\psi_n,\psi_1)_{L^2}=(\psi_n,\frac{\partial w_0}{\partial x_i})_{L^2}=0$ 
for $i=1,\cdots,N$ and $({\mathcal L}_0 \psi_n,\psi_n)_{L^2} \to 0$ as $n \to \infty$. 
Since $({\mathcal L}_0\psi_n,\psi_n)_{L^2} \to 0$ and $\| \psi_n \|_{L^2}=1$, 
one can show that $\| \psi_n \|_{\HN}$ is bounded. 
Thus passing to a subsequence, we may assume that 
$\psi_n \rightharpoonup \bar{\psi}$ in $\HN(\RN)$ and 
$\psi_n \to \bar{\psi}$ in $L_{{\rm loc}}^2(\RN)$ for some $\bar{\psi}\in \HN(\RN)$. 
Moreover arguing as \ef{eq:5.21}, we have 
\begin{equation}\label{eq:5.22}
\frac{1}{2} \| \psi_n \|_{\HN}^2
\le ({\mathcal L}_0\psi_n,\psi_n)_{L^2}
+\int_{\RN} \big( pw_0^{p-1} +32 w_0^2 |\nabla w_0|^2 \big) \psi_n^2 \,dx. 
\end{equation}
Since $w_0$ decays exponentially at infinity and $\| \psi_n\|_{L^2}=1$, 
there exists $R>0$ such that 
$$
\int_{B^c(0,R)} \big( pw_0^{p-1} +32 w_0^2 |\nabla w_0|^2 \big) 
\psi_n^2 \,dx \le \frac{1}{4}. 
$$
Thus from \ef{eq:5.22}, we get
$$
\frac{1}{4} \le C \int_{B(0,R)} \psi_n^2 \,dx +o_n(1).
$$
This implies that $\bar{\psi} \not\equiv 0$. 
Moreover by the Fatou lemma, the weak convergence of $\psi_n \rightharpoonup \bar{\psi}$, 
the strong convergence in $L^2_{{\rm loc}}(\RN)$ 
and by the exponential decay of $w_0$, one can show that
$$
({\mathcal L}_0 \bar{\psi},\bar{\psi})_{L^2} \le 0, \quad 
(\bar{\psi},\psi_1)_{L^2}=\Big( \bar{\psi}, \frac{\partial w_0}{\partial x_i} \Big)_{L^2}=0 
\quad \hbox{for} \ i=1,\cdots,N. 
$$
Since $\bar{\mu}=0$, 
it follows by the definition of $\bar{\mu}$ that $({\mathcal L}_0\bar{\psi},\bar{\psi})_{L^2}=0$. 
By the Lagrange multiplier rule, 
using $\bar{\psi},\psi_1$ and $\partial w_0/\partial x_i$ as test functions, 
one can prove that ${\mathcal L}_0\bar{\psi}=0$, which contradicts 
${\rm Ker} ({\mathcal L}_0)= {\rm span} \{ \frac{\partial w_0}{\partial x_i} \}$. 
Thus $\bar{\mu}>0$ and the proof is complete. 
\end{proof}

\begin{lemma}\label{lem:5.5}
It follows that $C_0=0$ and hence it holds 
$$
\phi(x,t)=\sum_{i=1}^N C_i \frac{\partial w_0}{\partial x_i}(x) +\tilde{\theta}(x,t). 
$$
\end{lemma}

\begin{proof}
First we observe by Lemma \ref{energy-est} that
\begin{equation}
\label{eq:5.23}
I(u)(t+t_n)-I(u)(t+\tau) \le 0, \quad \hbox{for any} \ 0<\tau \le t_n. 
\end{equation}
Let $t_0>1$ be given. From \ef{eq:5.2}, \eqref{gradienti} 
and uniform exponential decay of $\sup_{t>0} | D^ku(\cdot,t)|$ for $|k|\le 1$, one has 
$$
I(u)(t+t_n) \to I(w_0) \quad \hbox{as} \ n \to \infty \ \ \hbox{on} \ [1,t_0]. 
$$
Thus integrating \ef{eq:5.23} over $[1,t_0]$ and passing a limit $n \to \infty$, we get 
\begin{equation}\label{eq:5.24}
\int_1^{t_0} \big( I(u)(s+\tau)-I(w_0) \big) \,ds \ge 0 
\quad  \hbox{for any} \ t_0>1. 
\end{equation}
Next since $u_n=w_n+\eta_n \phi_n$, 
$I'(w_n)=0$ and $I(w_0)=I(w_n)$, by Taylor expansion, we have
\begin{align}\label{eq:5.25}
& \int_1^{t_0} \big( I(u)(s+t_n)-I(w_0) \big) \,ds \nonumber \\
&= \int_1^{t_0} \big( I(w_n+\eta_n \phi_n)-I(w_n) \big) \,ds \nonumber \\
&= \frac{\eta_n^2}{2} \int_1^{t_0}
\big\langle I''(w_n+\kappa_n \eta_n\phi_n)\phi_n,\phi_n \big\rangle \,ds
\nonumber \\
&= \frac{\eta_n^2}{2} \int_1^{t_0}
\big\langle I''(w_0)\phi_n,\phi_n \big\rangle \,ds
+\frac{\eta_n^2}{2} \int_1^{t_0}
\big\langle \big(I''(w_n)-I''(w_0) \big)\phi_n,\phi_n \big\rangle \,ds
\nonumber \\
&\quad +\frac{\eta_n^2}{2} \int_1^{t_0}
\big\langle \big( I''(w_n+\kappa_n \eta_n\phi_n)-I''(w_n) \big) \phi_n,\phi_n \big\rangle \,ds,
\end{align}
for some $\kappa_n \in (0,1)$. 
Now from $w_n+\kappa_n \eta_n \phi_n= \kappa_nu_n+(1-\kappa_n)w_n$, one has 
\begin{align*}
&\Big\langle \big( I''(w_n+\kappa_n \eta_n \phi_n )-I''(w_n) \big) \phi_n, \phi_n \Big\rangle \\
&= \int_{\RN} \Big\{ 
2 \phi_n^2 \Big( |\nabla \big( \kappa_n u_n+(1-\kappa_n ) w_n \big) |^2-|\nabla w_n|^2 \Big) 
+2 |\nabla \phi_n|^2 \Big( \big( \kappa_n u_n+(1-\kappa_n)w_n \big)^2-w_n^2 \Big) \\
&\hspace{4em} +8 \phi_n \nabla \phi_n \cdot \Big( \big( \kappa_n u_n + (1-\kappa_n)w_n \big) \nabla \big( \kappa_n u_n + (1-\kappa_n)w_n \big) -w_n \nabla w_n \Big) \\
&\hspace{4em} -p\phi_n^2 \Big( \big| \kappa_n u_n +(1-\kappa_n)w_n \big|^{p-1}-|w_n|^{p-1} \Big) \Big\} \,dx. 
\end{align*}
Since $u_n$ and $w_n$ converge to $w_0$ in $L^{\infty} \big( \RN \times [1,t_0] \big)$
by \ef{eq:5.2} and \ef{gradienti}, it follows that 
$$
\Big\langle \big( I''(w_n+\kappa_n \eta_n \phi_n)-I''(w_n) \big) \phi_n, \phi_n \Big\rangle
\le o(1) \| \phi_n(\cdot,t) \|_{H^1(\RN)}^2. 
$$
Thus by Lemma \ref{lem:5.1}, there exists $n_0=n_0(t_0) \in \N$ such that for $n \ge n_0$,
\begin{equation}\label{eq:5.27}
\int_1^{t_0}
\big\langle \big( I''(w_n+\kappa_n \eta_n\phi_n)-I''(w_n) \big) \phi_n,\phi_n \big\rangle \,ds 
\le 1.
\end{equation}
Similarly one gets for $n \ge n_1$, 
\begin{equation}\label{eq:5.26}
\int_1^{t_0} \big\langle \big(I''(w_n)-I''(w_0) \big)\phi_n,\phi_n \big\rangle \,ds \le 1. 
\end{equation}
Next since $\theta_n=\phi_n-\phi$, it follows that 
\begin{align*}
\int_1^{t_0}
\big\langle I''(w_0)\phi_n,\phi_n \big\rangle \,ds 
= \int_1^{t_0} \Big(
\big\langle I''(w_0)\phi,\phi \big\rangle
+2 \big\langle I''(w_0)\phi,\theta_n \big\rangle
+\big\langle I''(w_0)\theta_n,\theta_n \big\rangle \Big) \,ds.
\end{align*}
By (i) of Lemma \ref{lem:5.3}, there exists $n_1=n_1(t_0) \in \N$ such that for $n \ge n_1$,
\begin{align}\label{eq:5.28}
 \int_1^{t_0} \big\langle I''(w_0)\theta_n,\theta_n \big\rangle \,ds 
&= \int_1^{t_0} \int_{\RN}
\Big\{ (1+2w_0^2)| \nabla \theta_n|^2+8 w_0 \theta_n \nabla w_0 \cdot \nabla \theta_n \nonumber \\
&\hspace{6em} +2 \theta_n^2 |\nabla w_0|^2 +\theta_n^2 -pw_0^{p-1} \theta_n^2 
\Big\} \,dx \,ds \nonumber \\
&\le C \int_1^{t_0} \| \theta_n(\cdot ,s) \|_{H^1}^2 \,ds \le \tilde{C},
\end{align}
\begin{align}\label{eq:5.29}
 \int_1^{t_0} \big\langle I''(w_0) \phi,\theta_n \big\rangle \,ds 
&= \int_1^{t_0} \int_{\RN}
\Big\{ (1+2w_0^2) \nabla \phi \cdot \nabla \theta_n +4 w_0 \phi \nabla w_0 \cdot \nabla \theta_n  \nonumber \\
&\hspace{2em} 
+ 4 w_0 \theta_n \nabla w_0 \cdot \nabla \phi
+2 \phi \theta_n |\nabla w_0|^2 +\phi \theta_n -pw_0^{p-1} \phi \theta_n 
\Big\} \,dx \,ds \nonumber \\
&\le C \left( \int_1^{t_0} \| \phi(\cdot ,s) \|_{H^1}^2 \,ds \right)^{1\over 2}
\left( \int_1^{t_0} \| \theta_n(\cdot ,s) \|_{H^1}^2 \,ds \right)^{1\over 2}
\le \tilde{C},
\end{align}
where $\tilde{C}$ is independent of $n \in \N$ and $t_0$. 
Finally from \ef{eq:5.15}, it follows that 
\begin{align*}
\int_1^{t_0} \big\langle I''(w_0)\phi,\phi \big\rangle \,ds 
&= \big\langle I''(w_0)\psi_1,\psi_1 \big \rangle
\int_1^{t_0} (C_0)^2 e^{-2 \mu_1 s} \,ds \\
&\quad +2 \int_1^{t_0} C_0 e^{-\mu_1 s}
\Big\langle I''(w_0)\psi_1, \sum_{i=1}^N C_i 
\frac{\partial w_0}{\partial x_i} +\tilde{\theta}
\Big\rangle \,ds \\
&\quad +\int_1^{t_0} \Big\langle I''(w_0)
\Big( \sum_{i=1}^N C_i 
\frac{\partial w_0}{\partial x_i} +\tilde{\theta} \Big),
\sum_{i=1}^N C_i 
\frac{\partial w_0}{\partial x_i} +\tilde{\theta}
\Big\rangle \,ds.
\end{align*}
Noticing that 
$\big \langle I''(w_0)\psi_1,\, \cdot \, \big\rangle
=({\mathcal L}_0 \psi_1,\, \cdot \, )_{L^2}
=\mu_1(\psi_1,\, \cdot \, )_{L^2}$, 
$$
\Big\langle I''(w_0) \frac{\partial w_0}{\partial x_i},\, \cdot \, \Big\rangle
=\Big( {\mathcal L}_0 \Big( \frac{\partial w_0}{\partial x_i} \Big),\, \cdot \, \Big)_{L^2}=0,
$$
and that $\psi_1$, $\frac{\partial w_0}{\partial x_i}$ are orthogonal in $L^2$, we have 
\begin{align}\label{eq:5.30}
&\int_1^{t_0} \big\langle I''(w_0)\phi,\phi \big\rangle \,ds \nonumber \\
&= -\frac{(C_0)^2}{2} ( e^{-2\mu_1 t_0}-e^{-2\mu_1}) 
+2 \mu_1 C_0 \int_1^{t_0} e^{-\mu_1 s}
( \psi_1,\tilde{\theta} )_{L^2} \,ds
+\int_1^{t_0} \big\langle I''(w_0) \tilde{\theta},\tilde{\theta}
\big\rangle \,ds. 
\end{align}
Next by the Schwarz inequality, one has 
\begin{align*}
 2 \mu_1 C_0 \int_1^{t_0} e^{-\mu_1 s}
( \psi_1,\tilde{\theta} )_{L^2} \,ds 
&\le 2 | \mu_1| |C_0| e^{-\mu_1 t_0}
\| \psi_1 \|_{L^2} \int_1^{t_0} \| \tilde{\theta}(\cdot,s)
\|_{L^2} \,ds \\
&\le 2| \mu_1| |C_0| \sqrt{t_0}e^{-\mu_1 t_0}
\left( \int_1^{t_0} \| \tilde{\theta}(\cdot,s) \|_{H^1}^2 \,ds 
\right)^{1\over 2}.
\end{align*}
Let $\tilde{T}>0$ be as in Lemma \ref{lem:5.4} and suppose 
$\ell \tilde{T} \le t_0 \le (\ell +1)\tilde{T}$ for some $\ell \in \N \cup \{0\}$. Then 
\begin{align*}
\int_1^{t_0} \| \tilde{\theta}(\cdot,s) \|_{H^1}^2 \,ds
&\le \int_1^{\tilde{T}} \| \tilde{\theta}(\cdot,s) \|_{H^1}^2 \,ds
+\sum_{k=1}^{\ell} \int_{k\tilde{T}}^{(k+1)\tilde{T}} \| \tilde{\theta}(\cdot,s) \|_{H^1}^2 \,ds \\
&\le \int_1^{\tilde{T}} \| \tilde{\theta}(\cdot,s) \|_{H^1}^2 \,ds 
+\sum_{k=1}^{\ell} e^{- \alpha k\tilde{T}} \\
&\le \int_1^{\tilde{T}} \| \tilde{\theta}(\cdot,s) \|_{H^1}^2 \,ds
+\frac{e^{-\alpha \tilde{T}}}{1-e^{-\alpha \tilde{T}}}.
\end{align*}
Thus there exists $L>0$ independent of $t_0$ such that 
\begin{equation}\label{eq:5.31}
2\mu_1 C_0 \int_1^{t_0} e^{-\mu_1 s} 
(\psi_1,\tilde{\theta})_{L^2} \,ds
\le L| C_0 | \sqrt{t_0} e^{-\mu_1 t_0}.
\end{equation}
Similarly one has 
\begin{equation}\label{eq:5.32}
\int_1^{t_0} \big\langle I''(w_0) \tilde{\theta},\tilde{\theta} 
\big\rangle \,ds \le L. 
\end{equation}
From \ef{eq:5.25}-\ef{eq:5.32}, we obtain
\begin{align*}
&\int_1^{t_0} \big( I(u)(s+t_n)-I(w_0) \big) \,ds \\
&\le \frac{\eta_n^2}{2}
\Big( 2+2\tilde{C}+L+\frac{(C_0)^2}{2}e^{-2\mu_1}
+L |C_0| \sqrt{t_0} e^{-\mu_1 t_0}
-\frac{(C_0)^2}{2} e^{-2\mu_1 t_0} \Big)
\end{align*}
for $n \ge \max \{ n_0,n_1\}$. 
Now suppose by contradiction that $C_0 \ne 0$. 
Then since $\mu_1<0$, one has 
$$
2+2\tilde{C}+L+\frac{(C_0)^2}{2}e^{-2\mu_1}
+L |C_0| \sqrt{t_0} e^{-\mu_1 t_0}
-\frac{(C_0)^2}{2} e^{-2\mu_1 t_0}
\to -\infty \quad \hbox{as} \ t_0 \to \infty.
$$
This contradicts to \ef{eq:5.24}. 
Thus it follows that $C_0=0$ and hence the proof is complete. 
\end{proof}

\noindent
{\bf Proof of Proposition \ref{prop:4.2} concluded}. 
Let $C_1,\cdots, C_N$ be as defined in \ef{eq:5.14} 
and $\tilde{T}>0$, $\alpha>0$ be as in Lemma \ref{lem:5.4}. 
We put ${\bf C}=(C_1, \cdots, C_N)$ and $z_n:= \eta_n {\bf C} \in \RN$. 
Since $\eta_n \to 0$, we may assume $|z_n| \le 1$. 
By Lemma \ref{lem:5.5}, the orthogonality of $\frac{\partial w_0}{\partial x_i}$, $\tilde{\theta}(\cdot,\tau_0)$ in $L^2(\RN)$ and from $\| \phi_0 \|_{L^2} \le 2$, one has 
$$
4 \ge \| \phi_0\|_{L^2}^2
=\left\| \sum_{i=1}^N C_i \frac{\partial w_0}{\partial x_i}
+\tilde{\theta}(\cdot,\tau_0)
\right\|_{L^2}^2
= | {\bf C} |^2 \| \nabla w_0 \|_{L^2}^2
+\| \tilde{\theta}(\cdot,\tau_0) \|_{L^2}^2.$$
Since $\| \nabla w_0 \|_{L^2}=\| \nabla w\|_{L^2}$, it follows that
$| {\bf C} | \le \frac{2}{ \| \nabla w\|_{L^2}}=:M$ and
hence $| z_n| \le M \eta_n$.
Next by the definitions of $\phi_n$, $\theta_n$ and from Lemma \ref{lem:5.5},
we get
\begin{align*}
&u(x,t+t_n)-w(x+y_n+z_n)\\
&= \eta_n \phi_n(x,t)+w(x+y_n)-w(x+y_n+z_n) \\
&= \eta_n \theta_n(x,t)+\eta_n \phi(x,t)+w(x+y_n)-w(x+y_n+z_n)\\
&= \eta_n \theta_n(x,t)
+\sum_{i=1}^N \eta_n C_i \frac{\partial w_0}{\partial x_i}(x)
+\eta_n \tilde{\theta}(x,t)
+w(x+y_n)-w(x+y_n+z_n)\\
&= \eta_n\theta_n(x,t) +\eta_n \tilde{\theta_n}(x,t) 
+ \big( \nabla w_0(x) \cdot z_n + w(x+y_n) -w(x+y_n+z_n) \big) \\
&= \eta_n \theta_n(x,t) +\eta_n \tilde{\theta_n}(x,t) 
+ \big( \nabla w(x+y_n)\cdot z_n +w(x+y_n)-w(x+y_n+z_n) \big) \\
&\quad + \big( \nabla w_0(x) -\nabla w(x+y_n) \big) \cdot z_n. 
\end{align*}
By Lemma \ref{lem:2.5} (i), one has
$$
\| w(\cdot+y_n+z_n)-w(\cdot+y_n)-\nabla w(\cdot+y_n) \cdot z_n \|_{H^1}^2 
\le C |z_n|^4. 
$$
Moreover since $w(\cdot+y_n) \to w(\cdot +y_0)=w_0(\cdot)$ in $C^2(\RN)$, 
we also have
$$
\| \big( \nabla w_0(\cdot)-\nabla w(\cdot+y_n) \big) \cdot z_n \|_{H^1}^2
=o(1) |z_n|^2. 
$$
Thus by the Triangular inequality, we obtain
\begin{align*}
\eta^2(y_n+z_n,t_n+T) 
&= \int_T^{2T} \| u(\cdot,s+t_n)-w(\cdot+y_n+z_n) \|_{H^1}^2 \,ds \\
&\le \eta_n^2
\left( \int_{T}^{2T} \Big( 16 \| \theta_n(\cdot,s) \|_{H^1}^2
+ 16 \| \tilde{\theta}(\cdot,s) \|_{H^1}^2 \Big) \,ds
+CT \big( |z_n|^2+o(1) \big) \right). 
\end{align*}
By Lemma \ref{lem:5.3} (ii), there exist $T>1$ and $n_2=n_2(T) \in \N$ such that
$$
\int_T^{2T} \| \theta_n(\cdot,s) \|_{H^1}^2 \,ds \le \frac{1}{256} 
\quad \hbox{and} \quad CT \big( |z_n|^2+o(1) \big) \le \frac{1}{8} \quad \hbox{for} \ n \ge n_2.
$$
Taking $T>1$ large if necessary, we may assume 
$T>\tilde{T}$ and $e^{-\alpha T} \le \frac{1}{256}$.
Then by Lemma \ref{lem:5.4},
$$
\int_T^{2T} \| \tilde{\theta}(\cdot,s) \|_{H^1}^2 
\,ds \le \frac{1}{256}.$$
Thus we obtain $\eta^2(y_n+z_n,t_n+T) \le \frac{1}{4} \eta^2(y_n,t_n)$. 
This completes the proof. \qed

\section{Proof of the main results}

\noindent
In this section, we will prove the main results of the paper. 
\subsection{Proof of Theorem \ref{thm:1.2}}
Let $u_0 \in C_0^{\infty}(\RN)$ be non-negative, radially non-increasing and not identically zero. 
Then, by means of Lemma \ref{lem:0.3}, we know that 
$$
u(x,t) > 0, \quad 
u(x,t)=v(|x|,t), \quad 
v_r(|x|,t) < 0 \quad
\text{for any $x\in\R^N$ and $t\in (0,T_{\rm max})$}.
$$
If $u$ is globally defined, we have that $T_{\rm max}=\infty$. 
Then by Proposition \ref{prop:6.1}, we learn that $u$ is uniformly bounded in space and time, 
and it satisfies the decay condition \eqref{eq:1.4}. \qed

\subsection{Proof of Theorem \ref{thm:1.1}}
Let $w \in \Omega(u)$. 
Then there exists a diverging sequence $\{ t_n\}_{n\in\N}$ such that 
$u(\cdot,t_n)\to w(\cdot)$ uniformly in $\RN$ as $n \to \infty$. 
Let $T>1$, $\eta_0>0, \ t_0>0$ be as in Lemma \ref{lem:4.5} and fix $\varepsilon>0$. 
Then by (i) of Lemma \ref{lem:2.3}, there exists $n_0\in \N$ such that $t_{n_0} \ge t_0$, 
$$
\eta(0,t_{n_0})=
\left( \int_0^T \| u(\cdot,s+t_{n_0})-w(\cdot) 
\|_{H^1(\RN)}^2 \,ds \right)^{1\over 2}
< \min \{ \eta_0,\varepsilon \}.
$$
Thus from Lemma \ref{lem:4.5}, one has 
\begin{equation}\label{eq:4.4}
\eta(0,t_{n_0}+kT) \le \tilde{C} \varepsilon \quad \hbox{for every} \ k\in \N. 
\end{equation}
Let $t\ge t_{n_0}$ be given. Then it follows that 
$t_{n_0} +kT \le t\le t_{n_0} +(k+1)T$ for some $k\in \N$. 
Thus we can write $t=t_{n_0}+kT+\tau$ with $\tau \in [0,T]$. 
Then by Lemma \ref{lem:2.4} and from \ef{eq:4.4}, 
there exists $C>0$ independent of $t$ such that 
\begin{equation}\label{eq:4.5}
\eta(0,t)=\eta(0,t_{n_0}+kT+\tau)
\le C \eta(0,t_{n_0}+kT) \le C \varepsilon. 
\end{equation}
Now let $K>0$ be arbitrary. 
Then $\ell T \le K <(\ell+1)T$ for some $\ell \in \N$. 
Then from \ef{eq:4.5}, 
\begin{align*}
\int_0^K \| u(\cdot,s+t)-w(\cdot) \|_{H^1}^2 \,ds 
&\le \int_0^{(\ell+1)T} \| u(\cdot,s+t)-w(\cdot) \|_{H^1}^2 \,ds \\
&=\sum_{j=0}^\ell \eta^2(0,t+j T)\leq  (\ell+1) C^2 \varepsilon^2. 
\end{align*}
This implies that \ef{eq:1.5} holds. 
Finally we show that the limit $w\in \Omega(u)$ is independent of the
choice of the sequence $\{ t_n \}_{n\in\N}$. 
Indeed suppose that there exists another sequence $ \{ \tilde{t}_n \}_{n\in\N}$ such that $u(\cdot,\tilde{t}_n) \to \tilde{w}$ uniformly for some $\tilde{w} \in \Omega(u)$. 
Then by the previous argument, one has
$$
\int_0^K \| u(\cdot,s+t)-\tilde{w}(\cdot) \|_{H^1}^2 \,ds 
\le (\ell+1)C^2 \varepsilon^2.
$$
This implies that $w \equiv \tilde{w}$ and hence the proof is complete. 
\qed

\subsection{Proof of Theorem \ref{thm:1.3}}

Let $\varphi_0 \in C_0^{\infty}(\RN)$ be a function which is non-negative, 
radially non-increasing and not-identically equal to zero.
For $\lambda>0$, we denote by $u_{\lambda}$ the solution of \ef{eq:1.1}-\ef{eq:1.2} with the initial condition $u_0=\lambda \varphi_0$. 
Then two cases may occur, either $u_{\lambda}$ blows up in finite time or it is globally defined. 
In the second case, $u_{\lambda}$ is positive, radially decreasing and satisfies the uniform decay condition \ef{eq:1.4} by Lemma \ref{lem:0.3} and Theorem \ref{thm:1.2}. 
Thus by Theorem \ref{thm:1.1}, 
$u_{\lambda}$ converges to 0 or a positive solution of \ef{eq:1.3} uniformly in $\RN$. 
Now we define 
\begin{align*}
\CA&:= \{ \lambda \in (0,\infty) : \ u_{\lambda} \ 
\hbox{blows up in finite time} \}.\\
\CB&:= \{ \lambda \in (0,\infty) : \ u_{\lambda} \ 
\hbox{converges to a positive solution of \ef{eq:1.3}} \ \hbox{uniformly in $\RN$} \}.\\
\CC&:= \{ \lambda \in (0,\infty) : \ u_{\lambda} \ 
\hbox{converges to zero} \ \hbox{uniformly in $\RN$} \}.
\end{align*}
One can see that $\CA, \CB, \CC$ are intervals and 
$\CA \cup \CB \cup \CC=(0,\infty)$. The proof of Theorem \ref{thm:1.3} consists of four steps. 

\smallskip
\noindent
{\bf Step 1}: $\CA$ is open. 

By using standard parabolic estimates, one can prove that, for fixed $t_0>0$, the mapping
$$
\lambda \to I \big( u_{\lambda}(\cdot,t_0) \big)$$
is continuous. 
On the other hand, it follows from Lemmas \ref{lem:6.6} and \ref{lem:6.7} that 
$u_{\lambda}$ blows up in finite time 
iff there is $t_0>0$ such that $I \big( u_\lambda(\cdot,t_0) \big)<0$. 
These facts imply that $\CA$ is open.
\hfill$\Box$

\smallskip
\noindent
{\bf Step 2}: $\CC$ is open and not empty. 

We observe that any constant less than $1$ is a super-solution of \ef{eq:1.1}. 
Moreover as we have observed in the proof of Lemma \ref{lem:6.4}, 
any positive solutions of \ef{eq:1.3} have maximum values strictly larger than 1. 
Finally for fixed $t>0$, 
$u_{\lambda}(\cdot,t)$ is continuous with respect to $\lambda$ uniformly in $x\in \RN$. 
From these facts, one can show that $\CC$ is open and not empty. 
\hfill$\Box$

\smallskip
\noindent
{\bf Step 3}: $\CA$ is not empty.

We choose $R>0$ so that ${\rm supp} \ \varphi_0 \subset B(0,R)$. 
Then, taking into account Lemma \ref{lem:6.6}, 
it suffices to show that $I( \lambda \varphi_0)<0$ for large $\lambda>0$. 
It follows that
\begin{equation*}
I(\lambda \varphi_0) 
=\frac{\lambda^2}{2} \int_{B(0,R)} \big(|\nabla \varphi_0|^2+\varphi_0^2\big) \,dx 
+\lambda^4 \int_{B(0,R)} \varphi_0^2 |\nabla \varphi_0|^2 \,dx
-\frac{\lambda^{p+1}}{p+1} \int_{B(0,R)} |\varphi_0|^{p+1} \,dx.
\end{equation*}
If $p>3$, or $p=3$ and $\int_{\RN} \varphi_0^2 |\nabla \varphi_0|^2 
-\frac{1}{4} | \varphi_0|^4 \,dx <0$, 
then we have $I(\lambda \varphi_0) \to -\infty$ as $\lambda \to \infty$. 
Thus we have $I(\lambda \varphi_0)<0$ for large $\lambda>0$ and $\CA$ is not empty. 
\hfill$\Box$

\smallskip
Now since $(0,\infty)$ is connected, it follows that $\CB$ is not empty.

\smallskip
\noindent
{\bf Step 4}: $\CB$ consists of a single point $\lambda_0$. 

Suppose by contraction that the set $\CB$ has at least two elements $\lambda_0<\lambda_1$. 
We claim that $(\lambda_0,\lambda_1) \subset \CA$. 
Now let $\lambda \in (\lambda_0,\lambda_1)$ be arbitrarily given. 
First we show that
\begin{equation}\label{eq:7.1}
\lim_{ t \to T_{\lambda} } I \big( u_{\lambda}(\cdot,t) \big) <0,
\end{equation}
where $T_{\lambda}>0$ is the maximal existence time for $u_{\lambda}$. 
To this end, we suppose by contradiction that 
$I \big( u_{\lambda}(\cdot,t) \big) \ge 0$ for all $t\in (0,T_{\lambda}]$. 
Then by Lemmas \ref{lem:6.6}-\ref{lem:6.7}, $u_{\lambda}$ is globally defined. 
Moreover since $u_{\lambda_0}(\cdot,0) <u_{\lambda}(\cdot,0)$, 
we have $u_{\lambda_0}(x,t) \le u_{\lambda}(x,t)$ 
for all $x\in \RN$ and $t>0$ by the Comparison Principle.
Finally since $u_{\lambda_0}(x,t) \to w(x)$ as $t \to \infty$, 
it follows that $\lambda \in \CB$ and hence $u_{\lambda}(x,t) \to w(x)$ as $t \to \infty$ 
by the radial symmetry of $u_{\lambda}$ 
and the uniqueness of positive radial solution of \ef{eq:1.3}. 
Next we put $\phi=u_{\lambda}-u_{\lambda_0}$. Then from \ef{eq:1.1}, one has
\begin{align}\label{eq:7.2}
0&= \phi_t + {\mathcal L} \phi +2(u_{\lambda}^2-w^2)\Delta \phi 
+ \big( 4w\Delta w-2(u_{\lambda}+u_{\lambda_0})\Delta u_{\lambda_0} \big)\phi \nonumber \\
&\quad 
+2(|\nabla u_{\lambda}|^2-| \nabla w|^2)\phi
+\big( 4w \nabla w-2 u_{\lambda_0} \nabla (u_{\lambda}+u_{\lambda_0}) \big) \cdot \nabla \phi \\
&\quad 
+p \Big( w^{p-1}-\big( \kappa u_{\lambda}+(1-\kappa)u_{\lambda_0} \big)^{p-1} \Big) \phi, \nonumber
\end{align}
where $\kappa \in (0,1)$ and ${\mathcal L}$ is the linearized operator 
which is defined in \ef{linearized}. 
Let $\mu_1<0$ be the first eigenvalue of ${\mathcal L}_0$ 
and $\psi_1$ be the corresponding eigenfunction. 
Multiplying $\psi_1$ by \ef{eq:7.2} and integrating it over $\RN$, 
one can obtain as in the proof of Lemma \ref{lem:5.2} that 
$$
\frac{\partial}{\partial t} \int_{\RN} \phi \psi_1 \,dx
-\big( \mu_1 +\varepsilon(t) \big) \int_{\RN} \phi \psi_1 \,dx \ge 0 
\quad \hbox{for all} \ t>0,
$$
where $\varepsilon(t) \to 0$ as $t \to \infty$. 
Since $\mu_1<0$, it follows that 
$\int_{\RN} \phi(\cdot,t) \psi_1 \,dx \to \infty$ as $t \to \infty$. 
But this contradicts to $\phi(\cdot,t) \to 0$ as $t \to \infty$. 
Thus inequality \ef{eq:7.1} holds. 
Now from \ef{eq:7.1} and by the continuity of $I \big( u_{\lambda}(\cdot,t) \big)$ with respect to $t$, we have $I \big( u_{\lambda}(\cdot,t) \big)<0$ for $t$ sufficiently close to $T_{\lambda}$. 
Then one can show that $u_{\lambda}(x,t)$ blows up in finite time and hence $\lambda \in \CA$. 
Since $\lambda \in ( \lambda_0,\lambda_1)$ is arbitrarily, 
we obtain $(\lambda_0,\lambda_1) \subset \CA$ as claimed. 
Next for $\lambda \in (\lambda_0,\lambda_1)$, we have $u_{\lambda}(x,t) \le u_{\lambda_1}(x,t)$ for all $x\in \RN$ and $t>0$ by the Comparison Principle. 
Since $\lambda \in \CA$ and $\lambda_1 \in \CB$, it follows that
$u_{\lambda_1}(\cdot,t) \to w$ as $t \to \infty$ but $u_{\lambda}$ blows up in finite time. 
This is a contradiction and hence the set $\CB$ consists of a single point $\lambda_0$. 
Finally, by Steps 1-4, it follows that
$\CA=(\lambda_0,\infty)$,  $\CB=\{ \lambda_0\}$ and $\CC=(0,\lambda_0)$.
This completes the proof. \qed

\medskip

\end{document}